\documentclass[12pt,a4paper,reqno]{amsart}
\usepackage{amsmath,amssymb,graphics,epsfig,color,enumerate,mathrsfs, url}
 \usepackage{tikz}

\usepackage[T1]{fontenc}
\usepackage[utf8]{inputenc}

\usepackage{dsfont}  \usepackage{verbatim}\textwidth= 14. cm
\definecolor{refkey}{gray}{.75}
\definecolor{labelkey}{gray}{.5}

\newtheorem{Theorem}{Theorem}[section]

\newtheorem{Lemma}[Theorem]{Lemma}
\newtheorem{Proposition}[Theorem]{Proposition}
\newtheorem{Corollary}[Theorem]{Corollary}
\newtheorem{Remark}[Theorem]{Remark}

\newtheorem{Claim}[Theorem]{Claim}
\newtheorem{Definition}[Theorem]{Definition}

 \definecolor{darkgreen}{rgb}{0,0.6,0}

\definecolor{light}{gray}{0.9}   

\newcommand{\rosso}{\textcolor{black}}
\newcommand{\rot}{\textcolor{black}} 

\newcommand{\cA}{\ensuremath{\mathcal A}}
\newcommand{\cB}{\ensuremath{\mathcal B}}
\newcommand{\cC}{\ensuremath{\mathcal C}}
\newcommand{\cD}{\ensuremath{\mathcal D}}
\newcommand{\cE}{\ensuremath{\mathcal E}}
\newcommand{\cF}{\ensuremath{\mathcal F}}
\newcommand{\cG}{\ensuremath{\mathcal G}}

\newcommand{\cI}{\ensuremath{\mathcal I}}

\newcommand{\cK}{\ensuremath{\mathcal K}}
\newcommand{\cL}{\ensuremath{\mathcal L}}

\newcommand{\cP}{\ensuremath{\mathcal P}}
\newcommand{\cQ}{\ensuremath{\mathcal Q}}

\newcommand{\cU}{\ensuremath{\mathcal U}}

\newcommand{\cW}{\ensuremath{\mathcal W}}

\newcommand{\cZ}{\ensuremath{\mathcal Z}}

\newcommand{\bbE}{{\ensuremath{\mathbb E}} }

\newcommand{\bbL}{{\ensuremath{\mathbb L}} }

\newcommand{\bbN}{{\ensuremath{\mathbb N}} }

\newcommand{\bbP}{{\ensuremath{\mathbb P}} }
\newcommand{\bbQ}{{\ensuremath{\mathbb Q}} }
\newcommand{\bbR}{{\ensuremath{\mathbb R}} }

\newcommand{\bbZ}{{\ensuremath{\mathbb Z}} }

\newcommand{\vertiii}[1]{{\left\vert\kern-0.25ex\left\vert\kern-0.25ex\left\vert #1 
    \right\vert\kern-0.25ex\right\vert\kern-0.25ex\right\vert}}

\let\a=\alpha    \let\d=\delta  \let\e=\varepsilon
 \let\g=\gamma       \let\l=\lambda
      \let\o=\omega      
  \let\s=\sigma \let\t=\tau   
  \let\z=\zeta
\let\D=\Delta   \let\G=\Gamma  \let\L=\Lambda 
\let\O=\Omega      
\newcommand{\da}{\downarrow}

\newcommand{\be}{\begin{equation}}
\newcommand{\en}{\end{equation}}
\newcommand{\bem}{\begin{multline}}
\newcommand{\enm}{\end{multline}}
\newcommand{\bes}{\begin{equation*}}
\newcommand{\ens}{\end{equation*}}

\newcommand{\eccolo}{\cite[Lemma~8.2]{F_SEP}}

\author[A.~Faggionato]{Alessandra Faggionato}
\address{Alessandra Faggionato. Department of Mathematics, University La Sapienza, 
  P.le Aldo Moro 2, 00185 Rome, Italy}
\email{faggiona@mat.uniroma1.it}

\title[Simple exclusion  processes]{Graphical constructions of simple exclusion processes  with applications to      random environments}

\begin{document}
\begin{abstract}
We show that the symmetric simple exclusion process (SSEP) on \rot{a} countable  set is well defined by the stirring graphical construction as soon as the dynamics of a single particle is. The resulting process is  Feller, its  Markov generator is derived on local functions,   duality at the level of the  empirical density field holds. \rot{We also provide a general criterion assuring that local functions form a core for the generator}.
We then move to the simple exclusion process (SEP) and show that  the graphical construction leads to a well defined Feller  process under a percolation-type assumption corresponding to subcriticality in a percolation with random inhomogeneous parameters.  We derive its Markov generator on local functions \rot{which, under an additional general assumption,  form a core for the generator}. We discuss applications of the above results to SSEPs and  SEPs in random environments, where the standard assumptions to construct the process and investigate its basic properties (by the analytic approach or by graphical constructions)  are typically violated. As detailed in \cite{F_SSEP}, our results for SSEP  also allow  to extend the quenched hydrodynamic limit in path space obtained in \cite{F_SEP} by removing Assumption (SEP) used in there.

\smallskip


\noindent{\em MSC2020  Subject Classification}: 
60K35,  
60K37 
60G55, 
82D30 
\end{abstract}

\maketitle


\section{Introduction}
Given a countable set $S$ and given non-negative numbers $c_{x,y}$ associated to  $(x,y)\in S\times S$ with $x\not =y$, the simple exclusion process (SEP)  with rates $c_{x,y}$ is  the interacting particle system on $S$ roughly  described a follows. At most one particle can lie on a site and  each particle - when sitting at site $x$ - attempts to jump to a site $y$ with probability rate $c_{x,y} $, afterwards the jump is allowed only if the site $y$ is empty. One can think of a family of continuous-time random walks with jump probability rates $c_{x,y}$ apart from the hard-core interaction. The  SEP is called symmetric (and will be denoted as SSEP \rosso{below}) when $c_{x,y}=c_{y,x}$.
    Of course, conditions have to be imposed to have a well defined process for all times. For example, when the particle system is given by  a single particle, the random walk with jump probability rates $c_{x,y}$ has to be well defined for all times: the holding time parameter $c_x:=\sum_{y\in S:y\not =x} c_{x,y}$ is finite for all $x\in S$ and a.s. no explosion occurs (whatever the starting site is). 

The analytic approach in \cite{L1} assures that the SEP is well defined and is a Feller  process  with state space $\{0,1\}^S$ if 
\be\label{albero}
\sup_{x\in S} \sum_{y\in S:y\not =x} \max \{ c_{x,y}, c_{y,x}\}<+\infty\,.
\en
This follows by combining the Hille-Yosida Theorem (cf.~\cite[Theorem~2.9, \rosso{Chapter~1}]{L1}) with \cite[Theorem~3.9, \rosso{Chapter~1}]{L1}. Indeed,  conditions (3.3) and (3.8) in \cite[Theorem~3.9, \rosso{Chapter~1}]{L1} are both equivalent
 to \eqref{albero} \rosso{as derived in Appendix \ref{lingotto}}. For the SSEP, \eqref{albero} can be rewritten as $\sup_{x\in S} c_x <+\infty$.
It turns out that \eqref{albero} is a too much restrictive assumption when dealing with SEPs or SSEPs  in a  random environment, i.e.~with random  rates $c_{x,y}$ and possibly with a random set $S$. In this case we will write  $c_{x,y}(\o)$ and $S(\o)$  in order to stress the dependence from the environment $\o$.  
For example one could consider the SSEP on $S=\bbZ^d$   with nearest-neighbor jumps and i.i.d. unbounded jump probability rates associated to the undirected  edges.  Or,  starting with  a simple  point process\footnote{A simple point process on $\bbR^d$ is a random locally finite subset of $\bbR^d$ \cite{DV}.} $\o:=\{x_i\}$ on $\bbR^d$ (e.g.~a Poisson point process),  one could consider the SSEP on $S(\o):=\o$ with jump probability rates $c_{x,y}$ of the form $c_{x,y}=g(|x-y|)$ for all $x\not =y$ in $\o$ and for a fixed decaying function $g$. One could also consider the Mott variable range hopping (v.r.h.), without any mean-field approximation, which describes the electron transport in amorphous solids as doped semiconductors in the regime of strong Anderson localization at low temperature \cite{AHL,MA,Mott_Nob}. 
Starting with a marked simple point process $\o:=\{ (x_i,E_i)\}$ where\footnote{By definition of marked simple point process, $\{x_i\}$ is a simple point process and $E_i$ is called \emph{mark} of $x_i$ \cite{DV}} $x_i \in \bbR^d$ and $E_i \in [-A,A]$,  Mott v.r.h. corresponds to the SEP on $S(\o):=\{x_i\}$ where, for $i\not = j$, 
\[
c_{x_i,x_j}(\o):= \exp\{-|x_i-x_j| - \max\{E_j-E_i,0\} \}\,.
\]
The list of examples can be made \rot{much} longer \rosso{(cf.~e.g.~\cite{F_SEP,F_resistor,F_SSEP} for others)}.  The above   models anyway   show that, in some contexts with disorder, conditions \eqref{albero} is typically not satisfied (i.e. for almost all $\o$ \eqref{albero}  is violated with $S=S(\o)$ and $c_{x,y}=c_{x,y}(\o)$). On the other hand, it is natural to ask if a.s. the above SEPs exist, are Feller  processes, to ask how their Markov generator   behaves on good (e.g.~local) functions, \rot{when local functions form a core} and so on.

To address the above questions we leave the analytic approach of \cite{L1}  and move to graphical constructions. The graphical approach has a long tradition in interacting particle systems, in particular also for the investigation of attractiveness and duality. Graphical constructions of SEP and SSEP are  discussed  e.g.  in  \cite{H1,H2} (also for more general exclusion processes) and in  \cite[Chapter~2]{timo}, briefly in \cite[p.~383, Chapter~VIII.2]{L1} for SEP and \cite[p.~399, Chapter~VIII.6]{L1} for SSEP as stirring processes. For the graphical constructions of  other particle systems  we mention in particular \cite[Chapter~III.6]{L1} and \cite{Du10} (see also the references in \cite[Chapter~III.7]{L1}). On the other hand, the above references again make assumptions  not compatible with many applications to particles in a random environment (e.g.~finite range jumps or rates $c_{x,y}$ of the form $p(x,y)$, $p$ being  a probability kernel).

Let us describe our contributions. In Section \ref{sec_SSEP} we consider the SSEP on a countable set $S$ (of course, the interesting case is for $S$ infinite). Under the only assumption that the continuous time random walk on $S$ with jump probability rates $c_{x,y}$ is well defined for all times $t\geq 0$ (called Assumption SSEP below), we show that the stirring graphical construction 
  leads to a well defined Feller process, with the right form of generator on local functions \rot{and on other good functions}
     (see Propositions \ref{costruzione_SSEP}, \ref{feller_SSEP}, \ref{inf_SSEP}, \rot{\ref{eastpak}}
   in Section \ref{sec_SSEP}). \rot{We also provide a general criterion assuring that local functions form a core for the generator (see Proposition \ref{volpino})}. 
     Due to  its relevance for the study of hydrodynamics and fluctuations of the  empirical field for SSEPs in random environments (see e.g. \cite{CF,F_SEP,GJ}), we also investigate duality properties of the SSEP at the level of the empirical field (see Section \ref{sec_duality}). Finally, in Section \ref{SSEP_environment} we discuss some applications to SSEPs in a random environment. We point out that   the construction  of the SSEP on a countable set \rot{$S$}, when the single random walk is well defined, can be  obtained also  by duality and Kolmogorov's theorem as in \cite[Appendix~A]{CFRS}. On the other hand, the analysis there is   limited to the existence of the stochastic process.

 We then move to the SEP.   Under what we call Assumption SEP, which is inspired by Harris' percolation argument \cite{Du,H1}, in Section \ref{sec_SEP} we show that the graphical construction leads to a  well defined Feller process and derive explicitly its generator on local functions \rot{and other good functions (see Propositions \ref{costruzione_SEP},  \ref{feller_SEP}, \ref{inf_SEP},  \ref{eastpakbis})}. The analysis generalizes the one in \cite[Chapter~2]{timo} (done for $S=\bbZ^d$, and $c_{x,y}$ of the form $p(x-y)$ for a finite range probability $p(\cdot)$ on $\bbZ^d$).  
 Checking the validity of   Assumption SEP consists of  proving subcriticality in   suitable percolation models with  random inhomogeneous parameters.  \rot{Also for SEP we  provide a general criterion assuring that local functions form a core for the generator (see Propositions \ref{volpinobis}, \ref{boreale})}. 
 In Section \ref{SEP_environment} we discuss some applications to SEPs in a random environment. We point out that  in \cite{F_SEP} we assumed what we called there ``Assumption (SEP)'', which corresponds to the validity of the present 
  Assumption SEP for a.a. realizations of the environment.  In particular, in \cite{F_SEP} we checked its  validity for some classes of SEPs in a random environment. In Section \ref{SEP_environment} we recall these results   in the present language.  As a byproduct, we also  derive the existence (and several properties of its Markov semigroup on continuous functions) of Mott v.r.h. on a marked Poisson point process. 
 We point out that the  SEPs treated in \cite{F_SEP} are indeed SSEPs. As \rot{a byproduct of}   our results in Section \ref{sec_SSEP}, the quenched hydrodynamic limit derived in \cite{F_SEP} remains valid  also by removing Assumption (SEP) there. This application will be detailed in \cite{F_SSEP}.  \\

 \noindent
 \emph{Outline of the paper}. Section \ref{sec_noto} is devoted to notation and preliminaries. In Section \ref{sec_SSEP} we describe the stirring graphical construction and  our main results for SSEP (the analogous for SEP is given in Section \ref{sec_SEP}). In Section \ref{SSEP_environment} we discuss some applications to  SSEPs in a random environment (the analogous for SEP is given in Section \ref{SEP_environment}). The other sections \rosso{and Appendix \ref{lingotto}} are devoted to proofs.

 \section{Notation and preliminaries}\label{sec_noto}
Given a topological space $\cW$ we denote by $\cB(\cW)$ the $\s$--algebra of its Borel subsets. We think of $\cW$ as a measurable space with  $\s$--algebra of  measurable subsets given by  $\cB(\cW)$.

\smallskip

 Given a metric space $\cW$ with metric $d$ satisfying $d(x,y)\leq 1$ for all $x,y \in\cW$ \rosso{(this can be assumed at cost to replace $d$ by $d\wedge 1$)},  $D_{\cW}:=D(\bbR_+, \cW)$ is the space 
  of c\`adl\`ag paths from $\bbR_+:=[0,+\infty) $ to $\cW$  endowed with  the  Skorohod distance associated to $d$. We denote this distance by $d_{\rm S}$  and for completeness we recall its definition \rosso{(see \cite[Chapter~3]{EK})}.  Let $\L$ be the family of strictly increasing bijective functions $\l:\bbR_+\to \bbR_+$ such that  
\[
\g(\l):= \sup _{s>t \geq 0} \Big|\log \frac{\l(s)-\l(t)}{s-t}\Big|<+\infty\,.
\]
 Then, given $\z,\xi\in  D_\cW$, the distance $d_S(\z,\xi)$ is defined as 
 \[
 d_{\rm S} (\z, \xi ) :=\inf_{\l \in \L} \left\{
 \g(\l) \lor \int_0 ^\infty e^{-u}   \Big[   \sup_{t\geq 0} d\Big( \z(t \wedge u) , \z(\l(t) \wedge u) \Big)\Big]  du
  \right\}\,.
 \]
\rosso{Due to \cite[Proposition~5.3, Chapter~3]{EK}} $\z_n \to \z$ in $D_\cW$ if and only if there exists a sequence $\l_n\in \L$ such that 
\be
\g(\l_n) \to 0  \qquad \text{ and }\qquad  \sup_{0\leq t \leq T} d \big( \z_n(t) \,,\, \z(\l_n(t) )\big)\to 0 \; \; \text{for all } T>0\,. 
\en
If $\cW$  is  separable,  then 
the Borel $\s$--algebra $\cB(D_\cW)$ of
$D_\cW$  coincides with the $\s$-algebra generated by the coordinate maps $\z\mapsto \z(t)$, $t\in \bbR_+$ \rosso{(see \cite[Proposition~7.1, Chapter~3]{EK}).}
If $\cW$ is Polish \rot{(i.e.~it is a complete separable metric space)}, then also  $D_{\cW}$ is Polish \rosso{(see \cite[Theorem~5.6, Chapter~3]{EK}).}  

\smallskip

We now discuss two  examples, frequently used in the rest. 
In what follows we will take  $\bbN=\{0,1,2,\dots\}$ endowed with the discrete topology and in particular with distance  $ \mathds{1}(x\not =y)$ between $x,y\in \bbN$.  $D_\bbN$ will be endowed with the Skorohod metric, denoted in this case by $\mathfrak{d}$.  Since $\bbN$ is separable,   $\cB(D_{\bbN})$ is generated by the coordinate maps $D_{\bbN}\ni  \xi \mapsto   \xi (t) \in \bbN$, $t\in\bbR_+$.

\smallskip

Given a countable infinite  set $S$, we fix once and for all an enumeration
 \be S=\{s_n\,:\, n=1,2,\dots\}\label{enu}
 \en
  of $S$ and we endow $\{0,1\}^S$ with  the metric 
\be\label{sarto}
 d(\xi,\xi'):= \sum_{n\geq 1} 2^{-n} |\xi(s_n)-\xi'(s_n)|\,.
 \en
Then this metric induces the product topology on $ \{0,1\}^S$
($\{0,1\}$ has the discrete topology). \rot{We point out that  $\{0,1\}^S$ is a Polish space. Indeed, 
$\{0,1\}^S$ 
 is a compact metric space and therefore it is also complete and  separable. 
 As a consequence also $D_{\{0,1\}^S}$ is Polish}.
Given  a path $\xi \in D_{\{0,1\}^S}$ and a time $t\geq 0$,   we  will sometimes write   $\xi_t$ instead of $\xi(t)$. In particular,  $\xi_t(x)$  will be the value at $x\in S$ of the configuration $\xi_t$. Moreover, we will usually write $\xi_\cdot$ instead of $\xi$ to denote a generic path in $D_{\{0,1\}^S}$.

 Since  $ \{0,1\}^S$ is separable,  the $\s$--algebra $\cB\left(D_{\{0,1\}^S}\right)$ is generated by the coordinate maps  $D_{\{0,1\}^S}\ni \xi_{\cdot} \mapsto \xi_t \in \{0,1\}^S$, $t\in \bbR_+$.  Since $\cB\left( \{0,1\}^S\right)$ is generated by the coordinate maps $\{0,1\}^S \ni \xi \mapsto \xi(x) \in \{0,1\}$, we conclude that 
$\cB\left(D_{\{0,1\}^S}\right)$ is generated by the maps $D_{\{0,1\}^S}\ni \xi_{\cdot} \mapsto \xi_t  (x) \in \{0,1\}$ as $t,x$ vary in $\bbR_+ $ and $S$, respectively. This will be used in what follows.

\section{Graphical construction,  Markov generator and duality of SSEP}\label{sec_SSEP}
Let $S=\{s_n:n=1,2,...\}$ be an infinite  countable set. We denote by  $\cE_S$  the family of unordered pairs of elements of $S$, i.e.~
\be\label{logo_lego}\cE_S:=\{\{x,y\}\,:\, x\not = y,\; x,y\in S\}\,.
\en

To each pair $\{x,y\}\in\cE_S$ 
 we associate a number $c_{ \{x,y\} }\in [0,+\infty)$. To simplify the notation, we write $c_{x,y}$ instead of $c_{\{x,y\}}$. Note that 
 $c_{x,y}=c_{y,x}$ 
 for all $ x\not= y $ in $S$.
 Moreover,   to simplify the formulas below, we set
\[ 
c_{x,x}:=0 \qquad \forall x \in S\,.\]
  
 The following assumption, \rot{in force throughout all this section},   will assure that the graphical construction of the symmetric simple exclusion process (SSEP) as stirring process is well posed.
  
 \smallskip
 
 {\bf Assumption SSEP.} \emph{We assume that  the following two conditions are satisfied:
 \begin{itemize}
\item[(C1)] For all $x\in S$ it holds  $ c_x:=\sum_{y\in S} c_{x,y}\in [0,+\infty)$;
\item[(C2)] For each $x\in S$ the continuous-time random walk on $S$ starting at $x$ and with jump probability rates $\bigl( \rot{c_{y,z} \,:\, \{y,z\}}\in \cE_S \bigr)$ a.s. has no explosion.
\end{itemize}
}

\smallskip


When Conditions (C1) and (C2) are satisfied, we say that the random walk \rosso{$(X_t)_{t\geq 0}$} on $S$ with jump rates (or conductances) $c_{x,y}$ is well defined (for all times). This random walk (also called \emph{conductance model}, cf.~\cite{Bi})  is built in terms of waiting times and jumps as follows.  
Arrived at (or starting at) $x$, the random walk 
 waits at $x$ an exponential time  with mean $1/c_x\in (0,+\infty]$. If $c_x=0$,  this
 waiting time is infinite. If $c_x>0$, once completed its waiting,  the random walk jumps 
  to another site $y$ chosen with probability $c_{x,y}/c_x$ (independently from the rest).  Condition (C2) 
 says that, a.s., the jump times in the above construction have no accumulation point and therefore 
 the random walk is well defined for all times. 
 \smallskip 
  

\subsection{Graphical construction of the SSEP}\label{ape_maya}
 We  consider the product space $D_\bbN^{\cE_S}$ endowed  with the product topology (recall that $D_\bbN=D(\bbR_+,\bbN)$ is endowed with the Skorohod metric $\mathfrak{d}$, see Section \ref{sec_noto}). We write $\cK= ( \cK_{x,y} )_{\{x,y\}\in \cE_S}$ for a generic element of  $ D_\bbN^{\cE_S}$. 
 The product topology on $D_\bbN^{\cE_S}$ is induced by the metric 
  \be\label{pioggina}
  d( \cK, \cK'):= \sum_{\substack{i,j\in \bbN:\\1\leq i <j }} 2^{-(i+j)} \min \left\{1, \mathfrak{d}( \cK_{s_i,s_j}, \cK_{s_i,s_j}')\right\}\,.
  \en

\begin{Definition}[Probability measure $\bbP$] We  associate to each pair  $\{x,y\}\in \cE_S $    a Poisson process $( N_{x,y}(t))_{t\geq 0}$ with intensity $c_{x,y} $ \rosso{and with $N_{x,y}(0)=0$},  such that  the $N_{x,y}(\cdot)$'s  are independent processes when varying the pair $\{x,y\}$. We define $\bbP$ as the law on $D_{\bbN}^{\cE_S}$ of the   random object $( N_{x,y}(\cdot ) )_{\{x,y\}\in \cE_S}$ \rot{and we denote by $\bbE[\cdot]$ the expectation associated to $\bbP$}.
\end{Definition}

We stress that, since the pairs $\{x,y\}$ are unordered, we have $K_{x,y}=K_{y,x}$ and $N_{x,y}=N_{y,x}$.

\smallskip

We briefly recall the graphical construction  of the SSEP as stirring process (see also Figure \ref{fig1}). The detailed description and the proof that a.s. it is well posed will be provided in Section \ref{sec_proof_SSEP}.

 Given $\cK\in D_{\bbN} ^{\cE_S}$ and $x\in S$ we define   $X_t^{x}[\cK]$ 
 as the output of  the following algorithm (see Definition \ref{rumore} for details). 
  Start at $x$ and consider the set of  all   jump times  not exceeding   $t$ of the paths of the form  $\cK_{x,y}(\cdot)$ with  $y\in S$. If this set \rot{is} empty, then stop and define $X_t^x[\cK]:=x$, otherwise take the maximum value $t_1$ in this set.  If $t_1$ is the  jump time of $\cK_{x,x_1}(\cdot)$, then move to $x_1$ and 
  consider now the set of all    jump times  strictly smaller than    $t_1$ of the paths of the form  $\cK_{x_1,y}(\cdot)$ with  $y\in S$. If this set is empty,  then stop and define 
   $X_t^x[\cK]:=x_1$, otherwise take the maximum value $t_2$ in this set and repeat the above step. Iterate this procedure  until the set of jump times is empty. Then the algorithm stops and its output  $X_t^x[\cK]$ 
   is the last  site of $S$ visited by the algorithm. \rosso{Roughly, to determine $X_t^x[\cK]$ it is enough to follow the path in the graph of Figure \ref{fig1}, starting at $x$ at time $t$, going back in time and crossing an horizontal edge every time it appears. Then $X_t^x[\cK]$ is the site visited by the path at time $0$.}
     In Section \ref{sec_proof_SSEP} we will prove that the above construction is well posed  for all $x\in S$ and $t\geq 0$ if $\cK\in \G_*$, where $\G_*$ is a suitable Borel subset of  $D_{\bbN}^{\cE_S}$
 with $\bbP(\G_*)=1$.

\begin{figure}
\includegraphics[scale=0.4]{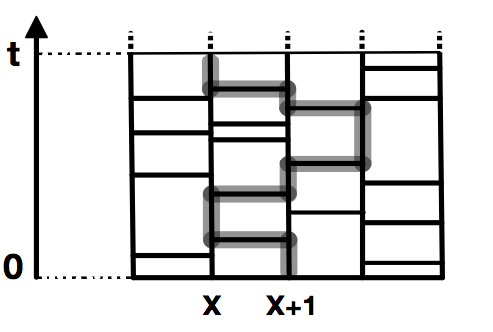}
\caption{Graphical construction of $X_t^{x}[\cK]$ when $S=\bbZ$ and   $c_{y,z}>0$ if and only if $|y-z|=1$.  $\cK$ is typical (jumps are only at edges $\{y,z\}$ with $|y-z|=1$). Vertical segments associated to the edge $\{y,z\}$ correspond to the jump times of  $\cK_{y,z}(\cdot)$. The vertexes $x_1,x_2,..$ built in the algorithm are the one visited by the bold path moving from time $t$ to time $0$. At the end one gets $X_t^{x}[\cK]=x+1$.} \label{fig1}
\end{figure}

  
 Having defined $X^x_t[\cK]$, given $\s\in \{0,1\}^S$ we set 
 \[ \eta^\s_t [\cK] (x):= \s\bigl( X^x_t[\cK]\bigr)\,.
 \] 
 Then $\eta^\s_t[\cK] \in \{0,1\}^S$. In Lemma \ref{pablito} in Section \ref{sec_proof_SSEP} we show  that   $\eta^\s_\cdot [\cK]= \bigl( \eta^\s_t[\cK]\bigr)_{t\geq 0}\in D_{\{0,1\}^S}$   and that the map $\G_* \ni \cK \mapsto  \eta^\s_\cdot [\cK]  \in D_{\{0,1\}^S} $
  is measurable in $\cK$. 

  We write $\cF$ for the $\s$--algebra $\cB ( D_{\{0,1\}^S })$ of Borel subsets of $D_{\{0,1\}^S }$. Since $\{0,1\}^S$ is separable, $\cF$ is 
  generated by the coordinate maps  $D_{\{0,1\}^S } \ni \eta \mapsto \eta_t\in \{0,1\}^S$, $t\geq 0$ (see Section \ref{sec_noto}).
  We also define
   $\cF_t$ as  the $\s$--algebra generated by the coordinate maps 
   $D_{\{0,1\}^S }\ni \eta \mapsto 
   \eta_s\in \{0,1\}^S$ with $s\in [0,t]$. Then $( D_{\{0,1\}^S }, (\cF_t)_{t\geq 0}, \cF)$ is a filtered measurable space.
  For each $\s\in \{0,1\}^S$ we define  $\bbP^{\s}$ as the probability measure on  the above filtered measurable space given by 
  \[
  \bbP^{\s} (A):= \bbP(\cK\in \G_*\,:\,  \eta^\s_\cdot [\cK]    \in A) \qquad A \in \cF\,.
  \]
  In what follows, we write $\bbE^\s$ for the expectation w.r.t. $\bbP^\s$.
   Similarly to \cite[Theorem~2.4]{timo} we get:
  \begin{Proposition}[Construction of SSEP] \label{costruzione_SSEP} The family $\big\{\bbP^\s:\s\in \{0,1\}^S\big\}$ of probability measures 
  on the  filtered measurable space $( D_{\{0,1\}^S }, (\cF_t)_{t\geq 0}, \cF)$
   is a Markov process (called \emph{symmetric simple exclusion process with conductances $c_{x,y}$}), i.e.
  \begin{itemize}
  \item[(i)] $\bbP^\s( \eta_0=\s)$ for all $\s \in \{0,1\}^S$;
  \item[(ii)] for any $A\in \cF$ the function $\{0,1\}^S\ni \s\mapsto \bbP^\s( A) \in [0,1]$ is measurable;
    \item[(iii)] for any $\s\in \{0,1\}^S$ and $A\in \cF$ it holds
    $\bbP^\s( \eta_{t+\cdot}\in A\,|\, \cF_t) = \bbP^{\eta_t}(A) $  $ \bbP^\s$--a.s.
  \end{itemize}
  \end{Proposition}
  For the proof of the above proposition see Section \ref{avatar}.

   \subsection{Markov semigroup and infinitesimal generator}   
  We write $C(\{0,1\}^S)$ for  the space of real continuous functions on $\{0,1\}^S$  endowed with the uniform norm.
  \begin{Proposition}[Feller property] \label{feller_SSEP}
Given $f\in C(\{0,1\}^S)$ and given $t\geq 0$, the map $S(t)f:\{0,1\}^S \to \bbR$ defined as $\big(S(t) f\big) (\s):= \bbE\left[ f\big( \eta^\s_t [\cK] \big)\right]
=\int d\bbP^\s(\eta_\cdot) f(\eta_t)$
 belongs to $C(\{0,1\}^S)$. In particular, the SSEP with conductances $c_{x,y}$ is a Feller process.
\end{Proposition}
For the proof of the above proposition see Section \ref{linus}.

\smallskip

Due to the Markov property in Proposition~\ref{costruzione_SSEP}, $(S(t))_{t\geq 0}$ is a   semigroup on $\{0,1\}^S$. Moreover, by 
using dominated convergence and that  $t\mapsto \eta^\s_t[\cK]$ is  right-continuous for $\cK\in \G_*$ (cf. Lemma \ref{pablito}), it is simple to check that $(S(t))_{t\geq 0}$ is a strongly continuous semigroup.
 Its infinitesimal generator $\cL$  is then the Markov generator $\cL$  of the SSEP with conductances $c_{x,y}$. We recall that $\cL$  has domain
\[
\cD(\cL ):=\big\{ f \in C(\{0,1\}^S)\,:\, \frac{ S(t) f- f }{t} \text{ has limit in }C(\{0,1\})^S\text{ as } t \da 0\big\}
\]
and is defined as $\cL f = \lim _{t\da 0} \frac{ S(t) f- f }{t}$ where the above limit is in $C(\{0,1\}^S)$.
  \begin{Proposition}[Infinitesimal generator \rot{on local functions}] \label{inf_SSEP}
 Local functions \rot{belong} to the domain $\cD(\cL)$. Moreover, for any local function $f$,   we have 
 \be\label{mammaE} 
 \cL f(\eta) = \sum_{x\in S} \sum_{y\in S}c_{x,y}  \eta(x) \bigl( 1- \eta(y)\bigr) \left[ f( \eta ^{x,y})- f(\eta)\right]\,,\;\; \eta \in \{0,1\}^{S}
 \en
and 
 \be\label{mahmood}
\cL f (\eta) = \sum _{\{x,y\}\in \cE_S} c_{x,y} \bigl[ f(\eta^{x,y})-f(\eta) \bigr]\,,\;\; \rot{\eta \in \{0,1\}^{S}} \,.
\en 
 The sum in the r.h.s. of \eqref{mammaE} and \eqref{mahmood} are absolutely convergent \rot{series of functions in $C(\{0,1\}^S)$}.
  \end{Proposition}

The configuration   $\eta^{x,y}$ is  obtained from $\eta$ by exchanging the occupation variables at $x$ and $y$, i.e.
\be\label{furia}
\eta^{x,y}(z)=\begin{cases}
\eta(y) & \text{ if } z=x\,,\\
\eta(x) & \text{ if } z=y\,,\\
\eta(z) & \text{ otherwise}\,.
\end{cases}
\en
\rot{Moreover}, we recall  that  a function $f:\{0,1\}^S\to \bbR$ is called \emph{local} if,  for some finite $A\subset S$, $f(\eta)$ is determined  by $(\eta_x)_{x\in A}$ (note that any local function is also continuous on $\{0,1\}^S$). \rot{In what follows, we will denote by $\cC$ the set of local functions, which is dense in $C(\{0,1\}^S)$.}

The proof of Proposition \ref{inf_SSEP}, given in Section \ref{manfred},  
has several similarities with then one in \cite[Appendix~B]{F_SEP}, where another graphical construction is used.

\rot{As in \cite{L1}, given $f\in C(\{0,1\}^S)$, we set 
\be\label{nero}
\D_f (x) := \sup\big\{ |f(\eta)- f(\xi)| : \eta,\xi\in \{0,1\}^S \text{ with } \eta(y)=\xi(y) \; \forall y\in S\setminus\{x\}\big\}\,.
\en
We also define 
\be\label{bianco}
\vertiii{f}:=\sum_{x\in S}  \D_f(x) \;\; \text{ and }\;\; \vertiii{f}_*:=
\sum_{x\in S} c_x \D_f(x) \,.
\en
Then Proposition \ref{inf_SSEP} can be extended to a larger class of functions by approximation. Indeed we have:
\begin{Proposition}[\rot{Infinitesimal generator on further good functions}] \label{eastpak}
Let $f\in C(\{0,1\}^S)$ satisfy $\vertiii{f}<+\infty$ and $\vertiii{f}_*<+\infty$. 
Then $f\in \cD(\cL)$ and 
$
\cL f(\eta)=  \sum _{\{x,y\}\in \cE_S} c_{x,y} \bigl[ f(\eta^{x,y})-f(\eta) \bigr]$,
where the r.h.s. 
 is an absolutely convergent series of functions in $C(\{0,1\}^S)$.
 \end{Proposition}
The proof of the above proposition is given in Section \ref{dim_eastpak}. Note that if $f$ is a local function, then $\D_f(x)=0$  except for a  finite set of elements $x$. In particular, local functions satisfy  both $\vertiii{f}<+\infty$ and $\vertiii{f}_*<+\infty$.}

\smallskip

\rot{For the next result  we recall that a set  $\cA$ is a core of $\cL$ if  $\cA\subset \cD(\cL)$ and the graph  of $\cL$ in $C(\{0,1\}^S)\times C(\{0,1\}^S)$ is the closure of the set $\{(f, \cL f )\,:\, f\in \cA\}$.   We denote by  $(X_t)_{t\geq 0}$ the continuous--time random walk on $S$ with jump probability rates $c_{x,y}$ (which is well defined by Assumption SSEP) and we denote by  $ E_x[\cdot]$   the expectation referred to the random walk $(X_t)_{t\geq 0}$ starting at $x$. Moreover,  we set $\bbQ_{+}:=\bbR_+\cap \bbQ$ and we recall that $c_x:=\sum_{y\in S} c_{x,y}$.
\begin{Proposition}[Core for $\cL$] \label{volpino} Suppose that 
\be \label{condo}
E_x[c_{X_t} ] <+\infty \qquad \forall t\in \bbQ_{+}\,, \; \forall x \in S\,.
\en
Then the family $\cC$ of local functions is a core for $\cL$.
\end{Proposition}
The proof of the above proposition is given in Section \ref{tana_volpino}.}

\begin{Remark}\label{silenziato}
\rot{Trivially, by combining Proposition  \ref{eastpak} and Proposition \ref{volpino}, we get that the set $\{f\in C(\{0,1\}^S)\,:\,\vertiii{f}<+\infty\,, \;\vertiii{f}_*<+\infty\}$ is a core for $\cL$ under condition \eqref{condo}.}
\end{Remark}

\rot{Our attention to deal in \eqref{condo} with $t\in \bbQ_+$ and not with $t\in \bbR_+$ is motivated by the applications to random walks in random environment (see Section \ref{SSEP_environment}). Indeed, dealing with the countable set $\bbQ_+$, \eqref{condo} is valid for almost any environment if, fixed $t\in \bbQ_+$, for almost any environment  it holds $E_x[c_{X_t} ] <+\infty$ for all $x\in S$.}

\rot{We note that, due to the symmetry $c_{x,y}=c_{y,x}$, condition \eqref{albero} assuring the validity of the analytic construction of the SSEP in \cite{L1} and in particular of \cite[Theorem~3.9]{L1} reads 
\be\label{cedro}
\sup_{x\in S} c_x <+\infty\,.
\en
When  \eqref{cedro} is satisfied,  \cite[Theorem~3.9]{L1} provides also a core for the generator of  SSEP, which is given by $\{ f\in C( \{0,1\}^S)\,:\, \vertiii{f}<+\infty\}$. It  is then standard to derive from this result that  $\cC$   is a core for $\cL$ (see e.g.  Remark \ref{batteria1000} in Section \ref{dim_eastpak} and use that the graph of $\cL$ is closed since  $\cL$ is  a Markov generator \cite[Chapter~1]{L1}).  All the above results from \cite{L1} are indeed included in ours since, under \eqref{cedro},  Assumption SSEP is fulfilled,  condition \eqref{condo} is automatically  satisfied and $\vertiii{f}_*<+\infty$ whenever $\vertiii{f}<+\infty$.}


\subsection{Duality \rosso{with the rw $(X_t)_{t\geq 0}$}}\label{sec_duality}   Due to  its relevance for the study of hydrodynamics and fluctuations of the  empirical field for SSEPs in random environment (see e.g.~\cite{CF,F_SEP,GJ}),  
we focus here  on  duality at the level of the density field. Apart from Lemma \ref{sfinimento}, the notions and results presented below are discussed in \cite[Sections~6 and 8]{F_SEP} (the proofs in \cite{F_SEP} can be easily adapted to our notation, since  it is enough to take $\e=1$ there and  to replace $\hat\o$ and $C_{\rm loc} (\e \hat \o)$  there by $S$ and $C_c(S)$ respectively).
 We denote by  $C_c(S)$ the set  of functions $f: S\to \bbR$ which are zero outside a finite subset of $S$. Since $S$ has the discrete topology, $C_c(S)$ corresponds to the set of real  functions with compact support.
We write $\mathfrak{n}$ for  the counting measure on $S$ and we introduce the set $\cD$ given by 
\[
\cD:=\big\{ f\in L^2(\mathfrak{n})\,:\, \sum_{x\in S}\sum_{y\in S} c_{x,y} \bigl( f(y)-f(x)\bigr) ^2 <+\infty\big\}\,.
\]
We then consider the bilinear form $\cE$ with domain $\cD$ given by 
\[
\cE (f,g):= \frac{1}{2}\sum_{x\in S}\sum_{y\in S} c_{x,y} \big( f(y)-f(x)\big) \big(g(y)-g(x)\big)\,, \qquad f,g\in \cD.
\]
On $\cD$ we introduce the norm  \rot{$\|f\|_\cD$ with $\|f\|^2 _{\cD}:=\|f\|^2_{L^2(\mathfrak{n})}+ \cE(f,f)$}. One can easily derive from 
  Condition (C1) that $C_c(S) \subset \cD$. We then  call   $\cD_*$ 
 the closure of $C_c(S)$ in $\cD$ w.r.t. the norm $\|\cdot \|_{\cD}$ (see \cite[Section~6]{F_SEP}). By \rot{arguing as in} \cite[Example~1.2.5]{FOT} the bilinear form $\cE$ restricted to  $\cD_*$ is a regular Dirichlet form. As a consequence, there  exists  a unique nonpositive self-adjoint operator $\bbL$ in $L^2(\mathfrak{n})$ such that $\cD_*$ equals the domain of $\sqrt{ - \bbL}$ and $\cE(f,f)= \|\sqrt{ - \bbL} f\|_{L^2(\mathfrak{n})}$ for any $f\in \cD_*$ (cf.~\cite[Theorem~1.3.1]{FOT}). By \cite[Lemma~1.3.2 and Exercise~4.4.1]{FOT} $\bbL$ is the infinitesimal generator of the strongly continuous Markov semigroup $(P_t)_{t\geq 0}$ on $L^2(\mathfrak{n})$ associated to the random walk $(X_t)_{t\geq 0}$ on $S$ 
with jump probability rates $c_{x,y}$ (defined in terms of holding times and jump probabilities as after Assumption SSEP). In particular, we have $P_t f (x) := E_x\bigl[ f(X_t)\bigr]$ where $E_x$ is  the expectation referred to the random walk starting \rot{at} $x$. In what follows we write $\cD(\bbL)$ for the domain of the operator $\bbL$. 
 
\begin{Definition}
Given a function $f:S\to \bbR$ such that $\sum _{x\in S}c_x |f(x)|<+\infty$, we define $\tilde{\bbL} f:S\to \bbR$ as $\tilde{\bbL} f(x):= \sum _{y\in S} c_{x,y} \big( f(y)-f(x) \big)$.
\end{Definition}
Note that 
 the series in the r.h.s. of the  definition of  $\tilde{\bbL} f(x)$ is absolutely convergent  by the assumption on $f$ and the symmetry of the $c_{x,y}$'s (indeed 
 $\sum _{y\in S} c_{x,y} |f(x)|= c_x |f(x)|<+\infty$, while  $\sum _{y\in S} c_{x,y} |f(y)|$ is finite since 
 we have $\sum _{x\in S}\sum _{y\in S} c_{x,y} |f(y)|=\sum _{x\in S}\sum _{y\in S} c_{y,x} |f(y)|=\sum _{y\in S} c_{y} |f(y)|$). In particular, if $f\in C_c(S)$ then $\tilde \bbL f $ is well defined.
 
 Although not necessary to prove the hydrodynamic limit of SSEPs on point processes, the following result has its own interest since it makes the generator  $\bbL$ explicit on local functions:
\begin{Lemma}\label{sfinimento} If $f\in C_c(S)$, then $f\in \cD(\rot{\bbL})$ and $\bbL f =\tilde\bbL f $.
\end{Lemma}
The above lemma is proved in Section \ref{lollo}.

We can now  describe the duality between the SSEP with conductances $c_{x,y}$ and the random walk with probability rates $c_{x,y}$ at the level of the density field (i.e.~the empirical measure).  To this aim we recall that given $\eta \in \{0,1\}^S$ the empirical measure $\pi[\eta]$ is the atomic measure on $S$ given by 
\be \label{def_empirical_meas} 
\pi[\eta] := \sum _{x\in S} \eta(x) \d_x \,.
\en
Given a real function $f$ on $S$ integrable w.r.t.~$\pi[\eta]$, we will write $\pi[\eta](f)$ or simply $\pi (f)$ for the sum $\sum _{x\in S}f(x) \eta(x)$  of $f$ w.r.t.~$\pi[\eta]$. 
Trivially, if $\sum_{x\in S} |f(x)|<+\infty$ as in the lemma below, then $f$ is  integrable w.r.t. $\pi[\eta]$  for all $\eta \in \{0,1\}^S$

Recall that $\cL : \cD(\cL) \to C(\{0,1\}^S)$ is the infinitesimal generator of the semigroup $( S(t))_{t\geq 0}$ on  $C(\{0,1\}^S)$ associated to the SSEP with conductances $c_{x,y}$.
\begin{Lemma}[see \eccolo] \label{dualita} Suppose that $f:S\to \bbR$ satisfies
\be\label{condizionale}
\sum_{x\in S} |f(x)|<+\infty\qquad  \text{ and } \qquad \sum_{x\in S} c_x |f(x) |<+\infty \,.
\en
Then the map $\{0,1\}^S \ni \eta \mapsto \pi[\eta](f)\in \bbR $ is continuous and indeed 
$\sum_{x\in S} f(x) \eta(x)$ is an absolutely convergent series in  $C(\{0,1\}^S)$. This map belongs to the domain $\cD(\cL)$ of $\cL$ and 
\be\label{danza}
\cL \big( \pi(f) \big) = \sum_{x\in S} \eta (x) \tilde{\bbL} f(x) \,,
\en
the r.h.s. of \eqref{danza} being an absolutely convergent series in $C(\{0,1\}^S)$.

If in addition to \eqref{condizionale} we have $f\in \cD(\bbL)\subset L^2(\mathfrak{n})$ (for example, if $f\in C_c(S)$), then $\bbL f=\tilde \bbL f$ and in particular we have the duality relation 
\be\label{danza_bis}
\cL \big( \pi(f)\big)= \sum_{x\in S}\eta(x) \bbL f(x)\,.
\en
\end{Lemma}
Identities of the form \eqref{danza_bis} are relevant to study hydrodynamics and fluctuations of the density field since they are associated to Dynkin's martingales. \rot{We point out that  above we have considered the Markov semigroup of the random walk in $L^2(\mathfrak{n})$ since particularly convenient for the stochastic homogenization analysis as in \cite{Fhom1}. Of course, one could have  considered as well  the Markov semigroup on other functional spaces, as the space $C_0(S)$ of continuous functions on $S$ vanishing at infinity endowed with the uniform norm}. 

\section{Applications to SSEPs in a random environment}\label{SSEP_environment}
We discuss some applications of the results presented in the previous section to SSEPs in a random environment.  \rot{We consider for simplicity $S=\bbZ^d$, but the arguments and results  we will present  can be  extended to more general graphs (see e.g. \cite{F_SSEP})}.

\smallskip

\rot{We take  $S=\bbZ^d$.}  We denote by $\bbE_d$ the set of undirected edges of the lattice $\bbZ^d$. We take  $\O:= \bbR_+^{\bbE_d}$, endowed with the product topology and the Borel $\s$--algebra \rosso{$\cB (\O)$}. We let  $\cP$ be a probability measure on $(\O,\rosso{\cB(\O)})$.   Given $\o\in \O$ we write $\o_{x,y}$ instead of $\o_{\{x,y\}}$.  Given $z\in \bbZ^d$ we write $\t_z:\O\to \O$ for the shift $(\t_z\o)_{ x,y } = \o _{x-z,y-z}$. 

 Given the generic environment $\o \in \O$, we set $c_{x,y}(\o):= \o_{x,y}$ if $\{x,y\} \in \bbE_d$ and  $c_{x,y}(\o):=0$ otherwise. \rosso{We write $(X_t^\o)_{t\geq 0}$ for the continuous-time random walk in the environment $\o$ with jump rates $c_{x,y}(\o)$ and with state space $\bbZ^d\cup \{\partial\}$, $\partial $ being a cemetery state (in case of explosion).}
\rosso{The properties stated in  Section \ref{sec_SSEP} hold $\cP$--a.s. if} Conditions (C1) and (C2) are satisfied $\cP$--a.s.. \rosso{Trivially} (C1) is always satisfied. For (C2) (i.e.~the random walk \rosso{$(X_t^\o)_{t\geq 0}$} a.s. does not explode) we have the following criterion:
\begin{Proposition}\label{visual}  Condition (C2)   is satisfied $\cP$--a.s.~in the following three cases:
\begin{itemize}
\item[(i)] $d=1$ and $\cP$ is stationary w.r.t.~shifts;
\item[(ii)]  $\cP$ is stationary  w.r.t.~shifts and $\int \cP(d \o) c_0(\o) <+\infty$;
\item[(iii)] $d\geq 2$ and under $\cP$ the coordinates $\o_{x,y}$, \rosso{as $\{x,y\}$ varies in $\bbE_d$}, are i.i.d.
\end{itemize}
\end{Proposition}
The proof of Item (ii) \rosso{will be}  an extension of the arguments used in  \cite[Lemma~4.3]{demasi} since we are not assuming here that $\cP$ is ergodic and that $\cP(\o _{x,y}>0)=1$ for all $\{x,y\}\in \bbE_d$.  Item (iii) \rosso{will follow} from the results of  \rosso{\cite{BD}}.

\begin{proof} We start with Item (i). Since $\cP$ is stationary, it is enough to restrict to the random walk starting at the origin.
Let $\o$ be an environment for which the random walk $X^\o_t$ starting at the origin has  explosion \rosso{in path space} with positive probability. Since the sum of infinite independent exponential variables with parameters upper bounded by a finite constant diverges a.s., if the random walk explodes with positive probability (in path space)
then the parameter $c_x(\o)$ has to diverge for $x\to +\infty$ or for $x\to -\infty$.
%
%

Now, given $M$, consider the random set $\hat \o_M:=\{x\in \bbZ\,:\, c_x(\o)\leq M\}$. 
For each $k\in \bbN$ consider the event $A_{M,k}\in \rosso{\cB(\O)}$ defined as $A_{M,k}:=\{\o \in \O\,: \max \hat \o_M=k\}$. By the stationarity of $\cP$, $\cP(A_{M,k})$ does not depend on $k$. On the other hand, the events $A_{M,k}$, $k\in \bbN$, are disjoint. Since $1\geq \cP( \cup_k A_{M,k})=\sum_{k}\cP( A_{M,k}) $ we conclude that $\cP( A_{M,k})=0$ for each $M$ and $k$.
\rosso{By reasoning similarly for $\min \hat \o_M$, we conclude that $\cP$--a.s. for any $M\in \bbN$  the set $\hat\o_M$ is empty or it is  unbounded from the left and from the right.}
This implies that $\cP$--a.s. $\lim_{x\to +\infty} c_x(\o)=+\infty$ is violated and, similarly, $\cP$--a.s. $\lim_{x\to -\infty} c_x(\o)=+\infty$ is violated. 
This allows to conclude that \rosso{for $\cP$--a.a.~$\o$ the random walk $X_t^\o$ a.s. does not explode}.
%
%
%
%
%

We now move to Item (ii). At cost to enlarge the probability space by marking 
edges by i.i.d. non degenerate random variables, independent from the rest, we can assume that 
\be
 \label{fiordalisi}
 \cP\big( \o\in \O\,:\, \t_z \o \not=\t_{z'}\o  \text{ for all } z\not=z' \text{ in } \bbZ^d\big)=1\,\rot{.}
 \en
 We let  $\cZ:=\int \cP(d \o) c_0(\o) <+\infty$. If $\cZ=0$ then $c_0=0$ $\cP$--a.s. and, by stationarity, $c_x=0$ for all $x\in \bbZ^d$ $\cP$--a.s.. In this case, $\cP$--a.s.,  all sites are absorbing for the random walk and therefore (C2) is trivially satisfied. From now on we restrict to the  case $\cZ>0$.   We define $\cQ$ as the probability measure on $\O$ given by  $\cQ(d \o)=\cZ^{-1} c_0 ( \o) \cP(d\o)$. 
 Due to  the stationarity of $\cP$, it is enough to consider the random walk starting at the origin \rosso{and show that a.s. there is no explosion}.

 We introduce the discrete-time Markov chain on $\O$ with jump probability rates $r(\o, \o')$ defined as  follows: if $c_0(\o)=0$ we set 
$r(\o,\o'):= \d_{\o,\o'}$,  if $c_0(\o)>0$ we set  
  $r(\o,\o'):= c_{0,x}(\o)/c_0(\o)$ when   $\o'=\t_x \o $ for some $x\in \bbZ^d$ and $r(\o,\o')=0$ otherwise. We write $(\bar \o_n)_{n\geq 0}$ for the above Markov chain when starting at $\o$. Note that, due to \eqref{fiordalisi}, $(\bar \o_n)_{n\geq 0}$ is well defined for $\cP$--a.a.~$\o$ and therefore for $\cQ$--a.a.~$\o$.
 Since $c_0(\o) r(\o,\o')=c_0(\o') r(\o',\o)$, the probability measure $\cQ$ is reversible for the above Markov chain. We write $P_{\cQ}$ for the law on the path space  $\O^{\bbN}$ of the Markov chain $(\bar \o_n)_{n\geq 0}$ with initial distribution $\cQ$. 
 Since $\cQ$ is reversible,  $P_{\cQ}$  is invariant w.r.t. shifts.
 
 We introduce now 
  a sequence  $(T_n)_{n\geq 0}$ of i.i.d. exponential times of mean one
  defined on another probability space $(\Theta, P)$. We can take $\Theta:= \bbR_+^{\bbN}$  with the 
  product topology, endowed with the Borel $\s$--algebra, and we can take 
  $T_k(\theta):=\theta_k$ for all $\theta \in \Theta$. Then $P_\cQ\otimes P$ is stationary w.r.t.~\rosso{time-}shifts when thought of  as probability measure on the path space $(\O\otimes \bbR_+)^{\bbN}$. We write $\cI$ for the $\s$--algebra of shift-invariant subsets of $(\O\otimes \bbR_+)^{\bbN}$.
   By the ergodic theorem    the limit 
  $ 
  \lim _{n \to _\infty} \frac{1}{n} \sum_{k=0}^{n-1} T_k(\theta)/c_0 (\bar \o _k)  
  $  exists $P_{\cQ}\otimes P$--a.s. and equals the expectation of $T_0(\theta)/c_0(\bar \o_0)$ w.r.t. 
  $P_{\cQ} \otimes P$ conditioned to $\cI$, which a random variable with values in $(0,+\infty]$. As a consequence, $P_{\cQ} \otimes P(\cW)=1$
  where  $\cW:=\{\sum_{k=0}^\infty T_k(\theta)/c_0 (\bar \o _k) =+\infty\}$. 
  We observe that $\cQ$  is concentrated  on  $\{\o:c_0(\o)>0\}$, and  $\cQ$ and $\cP$ are mutually absolutely continuous when restricted to this set. On the other hand, if $c_0(\o)=0$, we trivially have that $\rosso{\cP}_{\d_{\o}}\otimes P(\cW)=1$ \rosso{(the definition of $\cP_{\d_\o}$ is similar to the one of $\cP_\cQ$)}. We conclude that 
  $P_{\cP} \otimes P\,(\cW)=1$.  
  
  Finally we can build the continuous-time random walk $X^\o_\cdot=(X^\o_t)_{t\geq 0}$ with conductances $c_{x,y}(\o)$ and starting at the origin by  defining its jump process (i.e. the sequence of states visited by   $X^\o_\cdot $ in chronological order) as an  additive functional of the Markov chain 
  $(\bar \o_n)_{n\geq 0}$ (here we use again \eqref{fiordalisi}) and by  using the exponential times $T_k(\theta)/c_0 (\bar \o _k) $ as waiting times.
  The construction is standard: when the environment is $\o$, $X^\o_\cdot$ starts at the origin and remains there until time  $T_0(\theta)/c_0 (\bar \o _0) =T_0(\theta)/c_0 ( \o) $, afterwards it jumps to the site  $x\in \bbZ^d$ such that $\bar \o_1 =\t_x \o$ and remains there for a time   $T_1(\theta)/c_0 (\bar \o _1)$ and so on. By \eqref{fiordalisi}, for $\cP$--a.a.~$\o$ the above construction is well defined (e.g.~the above $x$ is univocally determined).
  The event $\cW$ then corresponds to non-explosion of the trajectory. Since $P_{\cP} \otimes P\,(\cW)=1$, we conclude that for $\cP$--a.a. $\o$ condition (C2) is fulfilled. 
  

We conclude with Item (iii). 
We   define $\G(\o)$ as the graph with edges $\{x,y\}\in \bbE_d$ with $\o_{x,y}>0$ and with vertexes given by the points belonging to the above edges. 
We write $V(\o)$ and  $E(\o)$ for the vertex  set and the edge set of  $\G(\o)$, respectively. For $e\in E(\o)$ we set $t(e):= \min\{1, \o_e^{-1/2}\}$ and, given $x,y\in V(\o)$, we set $\tilde d(x,y) :=\inf\{\sum _{i=1}^n t(e_i)\}$, where the infimum is taken over \rot{all} paths $(e_1,e_2,\dots, e_n)$ from $x$ to $y$ in $\G(\o)$. Then by  \cite[Lemma 2.5]{BD} the r.w. $X_t^\o$ a.s. does not explode if for any connected component $\cC$ of $\G(\o)$ there exists $x\in \cC$ and $\theta>0$ such that
\be\label{bingo12} \sum_{y\in \cC} \exp\{ -\theta \tilde d (x,y)\}<+\infty\,.
  \en
  We now define $\o'_e:= \max\{1, \o_e\}$ and     $t'(e):= \min\{1, (\o'_e)^{-1/2}\}$ for any $e\in \bbE_d$. Given 
  $x,y\in \bbZ^d$, we set $\tilde d'(x,y): =\inf\{\sum _{i=1}^n t'(e_i)\}$, where the infimum is taken over \rot{all} paths $(e_1,e_2,\dots, e_n)$ from $x$ to $y$ in the lattice  $\bbZ^d$.  For $e\in E(\o)$ we have $t(e)\geq t'(e)$ since  $  \o'_e \geq \o_e $. Since in addition paths in $\G(\o)$ are also paths in the lattice $\bbZ^d$, we get that
  $\tilde{d}'(x,y) \leq d (x,y)$ for any $x,y\in V(\o)$. In particular, 
  given $x\in \cC$ as in \eqref{bingo12}, the bound in \eqref{bingo12} is true if it holds
    \be\label{bingo13}
  \sum_{y\in \bbZ^d} \exp\{ -\theta \tilde{d}' (x,y)\}<+\infty\,.
  \en
  We observe that under $\cP$ the new conductances $\o'_{x,y}$, with $\{x,y\}\in \bbE_d$,  are i.i.d and lower bounded by 1. This is exactly the context of \cite{BD}. Then the bound \eqref{bingo13} holds for $\cP$--a.a.~$\o$   due to Theorem 4.3 and Lemma 2.11 in \cite{BD}.  We conclude that for $\cP$--a.a. $\o$ condition (C2) is fulfilled. 
\end{proof}

The \rot{result} in Proposition~\ref{visual}--(ii) is extended in \rot{\cite{F_SSEP}} to prove the a.s. non-explosion (i.e.~Condition  (C2)) for $\cP$--a.a.~$\o$ for a very large class of random walks on random graphs in $\bbR^d$, with also a random  vertex set (given by a simple point process).

\smallskip

\rot{We now give a simple criterion (based on Proposition \ref{volpino}) to verify that the set of local functions $\cC$ is a core for the generator $\cL=\cL(\o)$ for $\cP$--a.a.~environments $\o$: }

\begin{Proposition}\label{rev_core}
\rot{Suppose that $\cP$ is stationary w.r.t. shifts and $\int \cP(d \o) c_0 (\o)<+\infty$. Then for $\cP$-a.a. $\o$ condition \eqref{condo} is satisfied. In particular, for $\cP$--a.a.~$\o$  the family $\cC$ of local functions  is a core for the generator $\cL=\cL(\o)$.}
\end{Proposition}
\rot{We point out that the SSEP considered in Proposition \ref{rev_core} is well defined due to Proposition \ref{visual}--(ii).}
\begin{proof}
\rot{Given the environment $\o\in \O$, consider the process \emph{environment viewed from the particle} $(\bar \o_t)_{t\geq 0}$ , i.e.~$\bar \o_t:= \t_{X_t^\o}\o$. By the definition of the translations $\t_z:\O\to\O$ we have  $c_{X^\o_t}(\o)= c_0 ( \t_{X^\o_t}\o ) = c_0 (\bar \o_t)$. Moreover, because of the symmetry of the jump rates, we have that $\cP$ is a reversible (and therefore invariant) probability measure for the process. As a consequence,  we have
\be\label{gola}
\int \cP (d\o) {\rm E}_\o \big[ c_0 ( \bar \o_t )\big]= \int \cP (d\o)   c_0 (\o  )\,,
\en
where ${\rm E}_\o$ is the expectation w.r.t. the  process \emph{environment viewed from the particle} starting at the environment $\o$. By our assumption, \eqref{gola} is finite. This implies that, for any $t\geq 0$, ${\rm E}_\o \big[ c_0 ( \bar \o_t )\big]<+\infty$ for $\cP$--a.a.~$\o$. Hence, there exists $\cA\subset \O$ measurable with $\cP(\cA)=1$ such that, for all $\o \in \cA$, ${\rm E}_\o \big[ c_0 ( \bar \o_t )\big]<+\infty$ for any $t\in \bbQ_+$. This means that, for all $\o \in \cA$, $ E_0\big[c_{X^\o_t}(\o)\big]<+\infty$ for all $t\in \bbQ_+$. 
We set  $\cA_*:=\cap _{z\in \bbZ^d} \t_z \cA$. By the stationarity of $\cP$ we have $\cP(\t_z \cA)=\cP(\cA)=1$ and therefore $\cP(\cA_*)=1$. Moreover,  for all $\o\in\cA_*$ and $x\in \bbZ^d$, $\t_x\o\in \cA$ and therefore it holds  $ E_0\big[c_{X^{\t_x \o}_t}(\t_x \o)\big]<+\infty$ for all $t\in \bbQ_+$. Using that $ E_0 \big[c_{X^{\t_x \o}_t}(\t_x \o)\big]= E_x\big[c_{X^{ \o}_t}( \o)\big]<+\infty$, we get that \eqref{condo} is satisfied for all $\o \in \cA_*$.  Proposition \ref{volpino} allows to conclude. }
\end{proof}

\rot{The proof of Proposition \ref{rev_core} 
can be easily adapted to more general SSEP  in a random environment $\o$, with symmetric jump rates  $c_{x,y}(\o)$ and on  a random graph $\cG(\o)$ as in \cite{Fhom1,F_SSEP}. More precisely, when considering models as in \cite{Fhom1} with symmetric just rates, the condition assuring \eqref{condo} becomes  
 $\int \cP_0(d \o) c_0 (\o)<+\infty$ where $\cP_0$ is the Palm distribution associated to $\cP$ (of course, one does not need to require all the assumptions in \cite{Fhom1}).
 One can apply the above observation for example to  the 
   random walk  on the infinite cluster of a  supercritical  percolation on $\bbZ^d$ also with random conductances (assuring anyway stationarity). In this case $\cP_0$ would be  the probability measure  $\cP$ conditioned to the event that $0$ is in the infinite cluster (see \cite[Eq.~(12)]{Fhom1}). }

\section{Graphical construction and  Markov semigroup of SEP}\label{sec_SEP}
We discuss here the graphical construction and the Markov semigroup of the simple exclusion process (SEP)  on the  countable set $S$, when the jump rates $c_{x,y}$ are not necessarily symmetric. 


 We denote by  $\cE^o_S$  the family of ordered pairs of elements of $S$, i.e.~
\[\cE^o_S:=\{(x,y)\,:\, x\not = y,\; x,y\in S\}\,.
\]
To each pair $(x,y)\in \cE^o_S$ we associate a  number $c_{x,y}\in [0,+\infty)$. It is convenient to set \[\rot{c_{x,x}:=0 \qquad \forall x\in S}\,.
\] \rosso{Note that $c_{x,y}$ is not assumed to be symmetric in $x$, $y$.}

 We  consider the product space $D_\bbN^{\cE^o_S}$ endowed  with the product topology. This topology is induced by a
  metric $d(\cdot, \cdot)$, defined similarly to \eqref{pioggina}:
 $
  d( \cK, \cK'):= \sum_{i \geq 1}\sum_{j\geq 1} 2^{-(i+j)} \min \big\{1, \mathfrak{d}( \cK_{s_i,s_j}, \cK_{s_i,s_j}')\big\}$. 
We write $\cK= ( \cK_{x,y} )_{(x,y)\in \cE^o_S}$ for a generic element of  $ D_\bbN^{\cE^o_S}$.

\begin{Definition}[Probability measure $\bbP$] \label{def_taste} We  associate to each pair  $(x,y)\in \cE^o_S $    a Poisson process $( N_{x,y}(t))_{t\geq 0}$ with intensity $c_{x,y} $ \rosso{and with $N_{x,y}(0)=0$},  such that  the $N_{x,y}(\cdot)$'s  are independent processes when varying the pair $(x,y)$ in $\cE^o_S$. We define $\bbP$ as the law on $D_{\bbN}^{\cE^o_S}$ of the   random object $( N_{x,y}(\cdot ) )_{(x,y) \in \cE^o_S}$ \rot{and we denote by $\bbE[\cdot]$ the expectation associated to $\bbP$}.
\end{Definition}

The graphical construction of the SEP presented below 
 is based on Harris' percolation argument \cite{Du10,H1}. To justify this construction   we need a  percolation-type assumption. To this aim we define
 \[
 \cK_{x,y}^s(t) :=\cK_{x,y}(t) + \cK_{y,x}(t)\qquad \text{ for } t\geq 0\,,\; \{x,y\}\in \cE_S\,.
\]
\rosso{$\cE_S$ above is defined as in \eqref{logo_lego}}.
 Note the symmetry relation $\cK_{x,y}^s(t) =\cK_{y,x}^s (t)$ and that $\cK^s:=(\cK^s_{x,y})_{\{x,y\}\in \cE_S}$ belongs to $D_{\bbN}^{\cE_S}$.  When $\cK$ is sampled with distribution $\bbP$, 
 $\cK^s$  is a collection of independent processes and in particular   $\cK^s_{x,y}$ is a Poisson process with parameter 
 \[
   c^s_{x,y}:= c_{x,y}+ c_{y,x}\,.
   \]We set $c^s_{x,x}:=0$ for all $x\in S$.

 From now on we make the following assumption, \rot{in force throughout all this section}:
 
 \smallskip
  \noindent {\bf Assumption SEP}.
  \emph{There exists $t_0>0$ such that  
 for $\bbP$--a.a. $\cK\in D_{\bbN}^{\cE_S^o} $ 
 the undirected  graph $\cG _{t_0}(\cK)$ with vertex set $S$ and edge set 
 $\{ \{x,y\} \in \cE_S \,:\,  \cK^s_{x,y}(t_0)  > \rot{\cK^s_{x,y}(0)}\big\}$  has only connected components with finite  cardinality.}
 
 \smallskip
 \rot{We point out that for $\bbP$--a.a.~$\cK$ it holds $\cK_{x,y}(0)=0$ for all $(x,y)\in \cE_S^o$ and therefore $\cK_{x,y}^s(0)=0$ for all $\{x,y\}\in \cE_S$. Hence, Assumption SEP remains unchanged if we replace  $\cK^s_{x,y}(0)$ by zero there. On the other hand, the above choice is more suited for the construction of similar graphs for further time intervals as in Lemma \ref{lemma_prop_r} below.}
  Due to the loss of memory of the Poisson point process,  \rot{Assumption SEP}
 \rot{implies} the following property (we omit the proof since standard):
%

 \begin{Lemma}\label{lemma_prop_r} For $\bbP$--a.a. $\cK\in D_{\bbN}^{\cE_S^o}$ the following holds:
   $\forall r\in \bbN$ 
 the 
 undirected  graph $\cG ^r_{t_0}(\cK)$ with vertex set $S$ and edge set  $\{ \{x,y\}
 \in \cE_S 
 \,:\,  \cK^s_{x,y}((r+1) t_0)> \cK^s_{x,y} (r t_0) \}$ has only connected components with finite 
cardinality.
 \end{Lemma}
Trivially, $\cG^r_{t_0} (\cK)=\cG_{t_0} (\cK)$ for $r=0$.
We also  point out that   the properties appearing in Assumption SEP and Lemma \ref{lemma_prop_r} define indeed measurable subsets of $D_\bbN^{\cE^o_S}$:
\begin{Lemma}\label{susine}
 Given $r\in \bbN$ the set $\G_r$ of configurations $\cK\in D_{\bbN}^{\cE_S^o}$ such that   the graph $\cG^r_{t_0} (\cK)$  has only connected components with finite cardinality is a Borel subset of $D_{\bbN}^{\cE_S^o}$.
 \end{Lemma}
 \rosso{The proof of the above lemma is trivial and therefore omitted (simply note that $\G_r^c$ corresponds to the fact that, for some site $s_n$, for each integer $k\geq 1$ there exist distinct sites  $y_1, y_2, \dots, y_k $ in $S\setminus \{s_n\}$ such that $\cK^s_{y_i , y_{i+1}}((r+1) t_0)>\cK^s_{y_i , y_{i+1}}(r t_0)$ for all $ i=0,1,\dots, k-1$, where $y_0:=s_n$).}

 \medskip
 Trivially, Assumption SEP  can be reformulated as follows:
  
  \smallskip

 \noindent {\bf Equivalent formulation of Assumption SEP}: \emph{Given $t_0>0$, consider the  random graph with vertex set $S$ obtained by  putting an edge  between $x\not= y$ in $S$   with probability $1-\exp\{- c^s_{x,y}t_0\}$, independently when varying  $\{x,y\}$ among $\cE_S$. Then, for some $t_0>0$, the above random  graph 
 has a.s. only  connected components with finite  cardinality.}

  \begin{Remark}\label{dominio} By stochastic domination, to \rot{check} Assumption SEP one can as well replace  $\{ c^s_{x,y}\,:\, \{x,y\}\in \cE_S\} $
   by any other family  $\{ \bar{c}_{x,y}\,:\, \{x,y\}\in \cE_S\} $ such that $c^s_{x,y}\leq \bar c _{x,y}$ for any $\{x,y\}\in \cE_S$.
  \end{Remark}
 \smallskip

 In the above Assumption SEP we have not required any summability property as Condition (C1) in Assumption SSEP. Indeed, this is not necessary due to the following fact proved in Section \ref{sec_proof_SEP}:
 \begin{Lemma}\label{lemma_somma}Assumption SEP implies for all $x\in S$ that $
 c^s_x:=\sum_{y \in S } c^s_{x,y} <+\infty$,  $ \sum_{y \in S} c_{x,y} <+\infty$, $ \sum_{y \in S} c_{y,x} <+\infty$.
 \end{Lemma}

Recall the definition of $\G_r$ given in Lemma \ref{susine}.
\begin{Definition}[Set $\G_*$] \label{def_gamma_o}
We define $\G_*$ as the family of  $\cK\in D_{\bbN}^{\cE^o_S}$ such that 
\begin{itemize}
\item[(i)] $\cK\in \cap _{r\in \bbN} \G_r $;
\item[(ii)]
the  sum  $\sum_{y\in S\setminus x} \cK^s_{x,y} (t)=  \sum_{y\in S\setminus x} (\cK_{x,y} (t)+\cK_{y,x}(t) ) $ is finite for all  $x\in \bbN$ and $t\in \bbR_+$;
\item[(iii)]  given any $(x,y)\not=(x',y')$ in $\cE^o_S$ the set of  jump times of $\cK^o_{x,y}$ 
and the set of jump times of $\cK^o_{x',y'}$ 
 are disjoint and moreover all jumps \rot{equal} $+1$;
  \item[(iv)] \rot{$\cK_{x,y}(0)=0$ for all $(x,y)\in \cE_S^o$}.
   \end{itemize}
 \end{Definition}
It is simple to check (also by using  Lemmas  \ref{susine} and  \ref{lemma_somma})  the following:
\begin{Lemma}\label{muset_o}
$\G_*$ is  measurable, i.e.~$\G_*\in \cB(D_{\bbN}^{\cE^o_S}) $, and $\bbP(\G_*)=1$. 
\end{Lemma}

We briefly describe  the graphical construction  of the SEP for $\cK\in \G_*$  under Assumption SEP.  
 
  Given   $\s \in \{0,1\}^{S}$ we  first   define  a    trajectory $(\eta^\s_t[\cK])_{t \geq 0 }$  in $D_{\{0,1\}^{S} }$  starting at $\s$ by an iterative procedure.  We set 
  $\eta^\s_0[ \cK]:=\s$. 
   Suppose  that the   trajectory has been defined up to time $r t_0$, $r\in \bbN$.
   As $\cK \in \G_*$   all connected components of $\cG^r_{t_0}(\cK)$ have finite cardinality. 
    Let $\cC$ be 
    such a connected component and let
 \begin{equation}\label{grovis24}
   \begin{split}
     &\{s_1<s_2< \cdots <s_k\} =\\
  &\bigl \{s \,: \cK_{x,y}(s) = \cK_{x,y}(s-)+1 \text{ for some } x\not=y \text{ in } \cC, \; r t_0  <s \leq (r+1) t_0\bigr\}\,.
   \end{split}
 \end{equation}
The local evolution $\eta^\s _t[ \cK](z) $ with $z \in \cC$ and  $r t_0 < t  \leq (r+1) t_0$ is described as follows. 
 Start with $\eta^\s_{rt_0}[ \cK]$ as configuration at time $r t_0$ in $\cC$. At time $s_1$  move a particle from $x$ to $y\not=x$ with $x,y\in \cC$ if, just before  time $s_1$, it holds: \begin{itemize}
 \item[(i)] site $x$ is occupied and site $y$ is empty;
 \item[(ii)]  $\cK_{x,y}(s_1)= \cK _{x,y}(s_1-)+1$.
 \end{itemize}
Note that, since $\cK\in \G_*$, \rosso{the set \eqref{grovis24} is indeed finite and}  there exists  at most one ordered pair $(x,y)$ \rosso{satisfying (i) and (ii)}. After this first step, 
   repeat the same operation as above  orderly for times $s_2,s_3, \dots , s_k$. Then move to another connected component of  $\cG_{t_0}^r(\cK)$ and repeat the above construction and so on. 
   As 
the connected components are disjoint, the resulting path does not 
depend on the order by which we choose the connected components in the above algorithm (we could \rosso{as well}  proceed simultaneously with all connected components).
   This procedure defines $ (\eta^\s_t[\cK])_{ r t_0< t \leq (r+1) t_0}$.
  Starting with $r=0$ and progressively increasing $r$ by $+1$ we get the trajectory  $ \eta^\s _\cdot[\cK]= (\eta^\s_t[\cK])_{t\geq 0}$.  \smallskip

 The   filtered measurable space  $( D_{\{0,1\}^S }, (\cF_t)_{t\geq 0}, \cF)$ is defined as in Section \ref{ape_maya}. Again  the space $C(\{0,1\}^S)$ of real continuous functions    on $\{0,1\}^S$  is endowed with the uniform topology.
 Given  $\s\in \{0,1\}^S$, we define    $\bbP^{\s}$ as the probability measure on  the above filtered measurable space  given by 
 $\bbP^{\s} (A):= \bbP(\cK\in \G_*\,:\,  \eta^\s_\cdot [\cK]    \in A)$ for all $A\in\cF$.
  By Lemma \ref{pablito_o}   in Section \ref{sec_proof_SEP}  the set 
$\{ \cK\in \G_*\,:\,  \eta^\s_\cdot [\cK]    \in A \}$ is indeed measurable and therefore $\bbP^\s$ is well defined.


Similarly to Propositions \ref{costruzione_SSEP}, \ref{feller_SSEP} and \ref{inf_SSEP} we have the following results (see Section \ref{provette} and \ref{provette_bis} for their proofs): 
  \begin{Proposition}[Construction of SEP] \label{costruzione_SEP} The family $\big\{\bbP^\s:\s\in \{0,1\}^S\big\}$ of probability measures 
  on the  filtered measurable space $( D_{\{0,1\}^S }, (\cF_t)_{t\geq 0}, \cF)$
   is a Markov process (called \emph{simple exclusion process with rates $c_{x,y}$}), i.e.
  \begin{itemize}
  \item[(i)] $\bbP^\s( \eta_0=\s)$ for all $\s \in \{0,1\}^S$;
  \item[(ii)] for any $A\in \cF$ the function $\{0,1\}^S\ni \s\mapsto \bbP^\s( A) \in [0,1]$ is measurable;
    \item[(iii)] for any $\s\in \{0,1\}^S$ and $A\in \cF$ it holds
    $\bbP^\s( \eta_{t+\cdot}\in A\,|\, \cF_t) = \bbP^{\eta_t}(A) $  $ \bbP^\s$--a.s.
  \end{itemize}
  \end{Proposition}
 \begin{Remark} By changing $t_0$ in the graphs $\cG_{t_0}^r (\cK)$, for $\bbP$--a.a.~$\cK$ the     path $ \eta^\s _\cdot[\cK]$ constructed above does not change, and this  for any $\s$. In particular, the above SEP does not depend on the particular $t_0$ for which Assumption SEP holds.
 \end{Remark}
  \begin{Proposition}[Feller property] \label{feller_SEP}
Given $f\in C(\{0,1\}^S)$ and given $t\geq 0$, the map $S_t f (\s):= \bbE\left[ f\big( \eta^\s_t [\cK] \big)\right]
=\int d\bbP^\s(\eta_\cdot) f(\eta_t)$
 belongs to $C(\{0,1\}^S)$. In particular, the SEP  with rates $c_{x,y}$ is a Feller process.
\end{Proposition}
  \begin{Proposition}[Infinitesimal generator \rot{on local functions}] \label{inf_SEP}
 Local functions belong to the domain $\cD(\cL)$ of the infinitesimal generator $\cL$ of the  SEP  with rates $c_{x,y}$. Moreover, for any local function $f$,   we have 
 \be\label{mammaE_bis} 
 \cL f(\eta) = \sum_{x\in S} \sum_{y\in S}c_{x,y}  \,\eta(x) \bigl( 1- \eta(y)\bigr) \left[ f( \eta ^{x,y})- f(\eta)\right]\,,\;\; \eta \in \{0,1\}^{S}\,.
 \en
 The \rosso{series} in the r.h.s. is \rot{an  absolutely convergent series of functions in $C(\{0,1\}^S)$}.
  \end{Proposition}


\rot{We now show that Proposition \ref{inf_SEP} can be extended to a larger class of functions.
To this aim, given $f\in C(\{0,1\}^S)$, recall the definition of  $\D_f(x)$ given in  \eqref{nero} and recall that 
$\vertiii{f}:=\sum_{x\in S}  \D_f(x)$ (cf.~\eqref{bianco}). While in the symmetric case we considered 
$\vertiii{f}_*:= \sum_{x\in S} c_x \D_f(x)$, we now set
\[ \vertiii{f}_\star:=
\sum_{x\in S} c^s_x \D_f(x) \; \text{ where }\; c^s_x := \sum _{y\in S} c_{x,y}^s\,.
\]
Similarly to  Proposition \ref{eastpak} we have the following:
\begin{Proposition}[\rot{Infinitesimal generator on further good functions}] \label{eastpakbis}
Let $f\in C(\{0,1\}^S)$ satisfy $\vertiii{f}<+\infty$ and $\vertiii{f}_\star<+\infty$. 
Then $f\in \cD(\cL)$ and 
$
\cL f(\eta)=  \sum _{x\in S} \sum_{y\in S}c_{x,y} \eta(x) ( 1-\eta(y) )  \bigl[ f(\eta^{x,y})-f(\eta) \bigr]$,
where the r.h.s. 
 is an absolutely convergent series of functions in $C(\{0,1\}^S)$.
 \end{Proposition}
We point out that  local functions satisfy  both $\vertiii{f}<+\infty$ and $\vertiii{f}_\star<+\infty$, hence Proposition \ref{eastpakbis} is an extension of Proposition  \ref{inf_SEP}. The proof of Proposition \ref{eastpakbis} is given in Section \ref{cedrata}.
}

\smallskip

\rot{Finally we provide a criterion assuring that the family of local functions is a core for the generator $\cL$. To this aim we need the following:}
\begin{Definition}[Set $B_{r,x}(\cK)$ \rot{and $\cC_{t,x}(\cK)$}] \label{brunetta}
 Given $x\in S$, $r \in \bbN$ and $\cK\in \G_*$  we define the set $B_{r,x}(\cK)$ as follows. First we let $C_0$ be the connected component of $x$ in the graph $\cG^r _{t_0}(\cK)$. Then,  we  
 let $C_1$ be  the  union of the connected components in the graph $\cG^{r-1}_{t_0}(\cK)$ of $y$ as $y$ varies in $C_{0}$. In general, we 
  introduce iteratively $C_1,C_2,\dots, C_r$ by defining 
  $C_{j}$ as  the  union of the connected components in the graph $\cG^{r-j}_{t_0}(\cK)$ of $y$ as $y$ varies in $C_{j-1}$. We then set  $B_{r,x}(\cK):= C_r$ \rot{and $\cC_{t,x}(\cK):=B_{r,x}(\cK)=C_r$ for $rt_0<t\leq (r+1)t_0$, $\cC_{0,x}:=\{x\}$.}
  \end{Definition}
\rot{Properties and relevance of $B_{r,x}(\cK)$ and  $\cC_{t,x}(\cK)$ will be commented in Remark \ref{rem_latte} in Section \ref{sec_proof_SEP}.   We can now describe the above mentioned criterion:
  \begin{Proposition}[Core for $\cL$] \label{volpinobis} Suppose that 
\begin{align}
 &  \bbE\big[ | \cC_{t,x}|\big] <+\infty\qquad\;\; \;\,\forall x\in S\,, \forall t\in \bbR_+\,,\label{neve1}\\
 &  \bbE\big[ \sum_{z\in \cC_{t,x}} c_z^s \big]<+\infty \qquad \forall x\in S\,, \forall t\in \bbR_+ \,. \label{neve2}
 \end{align}
  Then the family $\cC$ of local functions is a core for $\cL$.
  \end{Proposition}
  The proof of Proposition \ref{volpinobis} is given in  Section \ref{astro}.
  }
  
\begin{Remark}\label{silenziatobis}
\rot{Trivially, by combining Proposition  \ref{eastpak} and Proposition \ref{volpino},  we get that the set $\{f\in C(\{0,1\}^S)\,:\,\vertiii{f}<+\infty\,, \;\vertiii{f}_*<+\infty\}$ is a core for  $\cL$ under conditions  \eqref{neve1} and \eqref{neve2}.}
\end{Remark}

\rot{The above conditions \eqref{neve1} and \eqref{neve2} correspond to   a countable family  of requests. Indeed, since 
  \eqref{neve1} and \eqref{neve2} are trivially satisfied for $t=0$ as  $\cC_{0,x}=\{x\}$, 
   \eqref{neve1} and \eqref{neve2}  are equivalent respectively to \eqref{danza1} and \eqref{danza2}:
   \begin{align}
 & \bbE\big[ | B_{r,x}(\cK)|\big] <+\infty \qquad \;\;\forall r\in \bbN\,,\;\forall x\in S\,,\label{danza1} \\
 & \bbE\big[ \sum_{z\in B_{r,x}(\cK)} c_z^s \big]<+\infty\qquad \forall r\in \bbN\,,\;\forall x\in S\,.
\label{danza2}
\end{align}
}

\smallskip

\rot{Below we write that $x\longleftrightarrow y$ in $\cG_{t_0}(\cK)$ if $\{x,y\}$ is an edge of $\cG_{t_0}(\cK)$. Since $\cC_{t_0, x}(\cK)$ corresponds to the connected component of $x$ in $\cG_{t_0}(\cK)$, we have $x\longleftrightarrow y$  if and only if 
$y\in \cC_{t_0, x}(\cK)$.
Given $x, y$ in $ S$, let 
\be\label{def_connie}
p(x,y):=\bbP( x\longleftrightarrow y \text{ in } \cG_{t_0})\,.
\en
The following result reduces the verification of \eqref{neve1} and \eqref{neve2} to a percolation problem associated to the random graph $\cG_{t_0} (\cK)$:} 
\begin{Proposition}\label{boreale}
\rot{Condition \eqref{neve1} is satisfied if   for all $x\in S$ and $ n\in \bbN_+$
\be\label{vento1}
\sum _{x_1,x_2,\dots, x_n\in S} p(x,x_1) p(x_1, x_2) \cdots p(x_{n-1}, x_n) <+\infty \,;
\en
equivalently if  for all $x\in S$ and $ n\in \bbN$
\be\label{vento1bis}
\sum _{x_1,x_2,\dots, x_n \in S}  p(x,x_1) p(x_1, x_2) \cdots p(x_{n-1}, x_{n}) \bbE\big[ | \cC_{t_0,x_n}| \big]<+\infty\,.
\en
Condition \eqref{neve2} is satisfied if  for all $x\in S$ and $n\in \bbN_+$
\be\label{vento2}
\sum _{x_1,x_2, \cdots ,x_n\in S} p(x,x_1) p(x_1, x_2) \cdots p(x_{n-1}, x_n)
c^s_{x_n}<+\infty\,;
\en
equivalently if for all $x\in S$ and $n\in \bbN $
\be\label{vento2bis}
\sum _{x_1,x_2, \cdots ,x_n\in S} p(x,x_1) p(x_1, x_2) \cdots p(x_{n-1}, x_n) \bbE\big[ \sum_{w\in  \cC_{t_0,x_n}} c_w^s \big]<+\infty\,.
\en
}
\end{Proposition}
\rot{Proposition \ref{boreale} is proved in Section \ref{aurora85}}. 

\rot{For $n=0$ \eqref{vento1bis} and \eqref{vento2bis} have to be thought of as  $\bbE\big[ | \cC_{t_0,x}| \big]<+\infty$ and $\bbE\big[ \sum_{w\in  \cC_{t_0,x}} c_w^s \big]<+\infty $.
Since $\sum_{y\in S} p(x,y)=| \cC_{t_0, x}|$, 
  \eqref{vento1bis} and \eqref{vento2bis} are automatically satisfied for all $x\in S$ and $n\in \bbN$  if 
\be \label{lucina} \sup_{x\in S} \bbE\big[ | \cC_{t_0,x}| \big]<+\infty \;\; \text{ and }\;\; \sup_{x\in S} \bbE\big[ \sum_{w\in  \cC_{t_0,x}} c_w^s \big]<+\infty \,.
\en 
If  $\sup_{x\in S} c_x^s<+\infty$, then \eqref{lucina} reduces to 
$\sup_{x\in S} \bbE\big[ | \cC_{t_0,x}| \big]<+\infty  $ (this is indeed what checked in \cite{timo} in order to prove 
that $\cC$ is a core for $\cL$ for the exclusion process on $\bbZ^d$  considered in \cite[Proposition~4.5]{timo}).}


 
 \section{Applications to SEPs in a random environment} \label{SEP_environment}
 We now consider the SEP with state space $S(\o)$    and jump probability rates $c_{x,y}(\o)$ depending on a random environment $\o$. More precisely,  we have a probability space  $(\O,\cG,
 \cP)$ and the environment $\o$ is a generic element of $\O$.

Due to the results presented in Section \ref{sec_SEP}, to have  $\cP$--a.s.~a well defined process built by the graphical construction and  enjoying the properties stated in Propositions \ref{costruzione_SEP}, \ref{feller_SEP} and \ref{inf_SEP},  it  is sufficient to check  that for  $\cP$--a.a.~$\o$ Assumption SEP is valid with $S=S(\o)$ and $c^s_{x,y}=c^s_{x,y}(\o):= c_{x,y}(\o) +c_{y,x}(\o)$.
 This becomes an interesting problem in percolation theory \rosso{since one has a percolation problem in a random environment}. 
By restating the content of Propositions 5.1 and 5.2 in \cite{F_SEP} in the present setting ($c^s_{x,y}(\o)$ here plays the role of the conductance of $\{x,y\}$ there)  we have the two  criteria below, where  $\bbE^o_d$ ($\bbE_d$) denotes  the set of directed (undirected) edges of the lattice $\bbZ^d$.

\begin{Proposition}\cite[\rot{Proposition}~5.1]{F_SEP} \label{cuore25}
 Let $  \O:=[0,+\infty)^{\bbE^o_d}$. Suppose that  $\cP$ is stationary w.r.t. shifts. Take $S( \o) :=\bbZ^d$ and $c_{x,y}(\o):= \o_{x,y}$ for  $(x,y)\in \bbE^o_d$ and $c_{x,y}(\o):=0$ otherwise. Then for $\cP$--a.a.~$\o$ Assumption SEP is satisfied if at least 
  one of the following conditions is fulfilled: 
\begin{itemize}
\item[(i)]   $\cP$--a.s. there exists a constant $C(\o)$ such that $\o_{x,y}\leq C(\o)$ for all $(x,y)\in \bbE^o _d$; 
\item[(ii)] under $\cP$ the  random variables   $c^s_{x,y}(\o) =\o_{x,y}+\o_{y,x}$  are independent when varying $\{x,y\}$  in $\bbE_d$;
\item[(iii)] for some $k>0$ under $\cP$ the random variables $c^s_{x,y}(\o)=\o_{x,y}+\o_{y,x}$   are $k$--dependent  
 when varying $\{x,y\}$  in $\bbE_d$.
\end{itemize}
\end{Proposition}
The \rosso{$k$--dependence} in Item (iii) means that, given $A,B \subset \bbZ^d$ with distance at least $k$, the random fields $( c_{x,y}^s (\o)\,: x,y \in A,\, \{x,y\}\in \bbE_d)$ and $( c_{x,y}^s (\o)\,: x,y \in B,\, \{x,y\}\in \bbE_d)$ are independent \rosso{(see e.g. \cite[Section~7.4]{Gr})}.

\begin{Proposition} \cite[\rot{Proposition}~5.2]{F_SEP}
Suppose that there is a measurable map $\o \mapsto \hat \o$ into the set of locally finite subsets of $\bbR^d$, such that  $\hat \o$  is a Poisson point process  (PPP) when $\o$ is sampled  according to $\cP$.  Take $S(\o):=\hat \o$. Suppose  that, for $\cP$--a.a. $\o$, 
 $c^s_{x,y}(\o) \leq g(|x-y|)$   for any $x,y \in S(\o)$, where $g(r)$ is a fixed bounded function such that 
 the map $ x \mapsto g(|x|) $ belongs to  $L^1(\bbR^d, dx)$. Then 
for $\cP$--a.a. $\o$ 
Assumption SEP is satisfied.
\end{Proposition}
 By taking $g(r)=2 e^{-r}$ the above proposition implies the following (recall Mott v.r.h. discussed in the Introduction):
 \begin{Corollary}
 For $\cP$--a.a. $\o$ the  Mott v.r.h. on a marked Poisson point process (without the mean field approximation) can be built  as SEP by the graphical construction and  it is a  Feller process, whose Markov generator is given by  \eqref{mammaE_bis} on local functions.
\end{Corollary}

\rot{We now give an application of Proposition \ref{boreale} to check that $\cC$ is a core for the generator. We discuss a special case, where \eqref{lucina} can be  violated (of course, further generalizations are possible, we just aim to illustrate the criterion):
\begin{Proposition} Let $  \O:=[0,+\infty)^{\bbE^o_d}$.   Take $S( \o) :=\bbZ^d$ and $c_{x,y}(\o):= \o_{x,y}$ for  $(x,y)\in \bbE^o_d$ and $c_{x,y}(\o):=0$ otherwise.
Suppose  that  the random variables $c^s_{x,y}(\o)= \o_{x,y}+ \o_{y,x}$ with $\{x,y\}\in \bbE_d$ are i.i.d. and have  finite $(1+\e)$-moment for some $\e>0$ (i.e. the expectation of $c^s_{x,y}(\o)^{1+\e}$ is finite). 
Then, for $\cP$--a.a. environments $\o$,  the  family  $\cC$ of local functions is a core for the generator of the SEP.\end{Proposition}
We point out that for the above model Assumption SEP is satisfied due to Proposition \ref{cuore25}.}

\begin{proof} \rot{Given the environment $\o$, we write $\bbP_\o$ and $p_\o(x,y)$  instead of $\bbP$ and $p(x,y)$ (cf.~Definition~\ref{def_taste} and Eq.~\eqref{def_connie}) in order to stress the dependence on $\o$. Let $p_c$ be the critical probability for the Bernoulli bond percolation on $\bbZ^d$. $p_c=1$ for  $d=1$, while $p_c\in (0,1)$ for $d\geq 2$. In both cases  we can fix $a>0$ such that $\cP( c^s_{x,y}\geq a) < p_c/2$ for  all $\{x,y\}\in \bbE_d$.  Given $(\o,\cK)$  we consider the undirected graph 
$\bar\cG_{t_0}(\o, \cK)$ with vertex set $\bbZ^d$ and edges given by the pairs $\{x,y\}\in \bbE_d$
such  that (i) $c^s_{x,y} (\o) \geq a$ or (ii) $c^s_{x,y} (\o) <a$ and $\cK^s_{x,y}(t_0)> \cK^s_{x,y}(0)$.  When  
$\{x,y\}\in \bbE_d$ is an edge of $\bar\cG_{t_0}(\o, \cK)$ we say that it is open, otherwise we say that it is closed.  Note that $\cG_{t_0} (\cK)$ is a subgraph of $\bar \cG_{t_0} (\o,\cK)$.
Moreover, under $\cP(d\o)  \otimes \bbP_\o$, the random graph $\bar \cG_{t_0} (\o,\cK)$  corresponds to a Bernoulli bond percolation on $\bbZ^d$ \cite{Gr} with parameter 
\[ p:= \cP ( c^s_{x,y} (\o) \geq a) + \int \cP (d\o) \mathds{1}( c^s_{x,y} (\o) < a) \left( 1- e^{- c_{x,y}^s(\o) t_0}\right)<  \frac{p_c}{2} +  1- e^{-a  t_0}\,.
\]
Since $a>0$ we can fix $t_0>0$ small enough to assure that $p<p_c$. 
As a consequence,  $\cP(d\o)  \otimes \bbP_\o$--a.s. the graph $\bar \cG_{t_0} (\o, \cK)$ is subcritical.
Due to the results of subcritical Bernoulli bond percolation (cf.~\cite[Theorem~(5.4)]{Gr}),  there exists $c>0$ such that 
the $\cP(d\o)  \otimes \bbP_\o$--probability that two points $x,y\in \bbZ^d$ are connected in 
$\bar \cG_{t_0} (\o,\cK)$ is bounded by $e^{-c |x-y|}$.  Hence also the  $\cP(d\o)  \otimes \bbP_\o$--probability that two points $x,y\in \bbZ^d$ are connected in 
$ \cG_{t_0} (\cK)$ is bounded by $e^{-c|x-y|}$, i.e.~
\be\label{chiave74}
{\rm E}\left[ p_\o(x,y)\right] \leq e^{-c|x-y|} \qquad  \forall x,y\in S\,,
\en
where ${\rm E}[\cdot]$ denotes the expectation w.r.t.~$\cP$.}

\rot{To get that $\cC$ is a core for $\cL=\cL(\o)$ $\cP$--a.s.  we apply Proposition \ref{boreale}.
Since we just need  that $\cP$--a.s. the  countable family of conditions \eqref{vento1} with $x\in S$ and $n\in \bbN_+$ holds, it is enough to prove that given $x\in S$ and $n\in \bbN_+$ condition \eqref{vento1} holds $\cP$--a.s. and to this aim it is enough to show that 
\be\label{vento1tris}
\sum _{x_1,x_2,\dots, x_n\in S} {\rm E} \left[ p_\o(x,x_1) p_\o(x_1, x_2) \cdots p_\o(x_{n-1}, x_n) \right]<+\infty \,.
\en
By applying several times Schwarz inequality, using that $p_\o(a,b)^\a \leq p_\o(a,b)$ for $\a\geq 1$ since  $p_\o(a,b)\in [0,1]$ and using \eqref{chiave74}, as detailed in an example below, we can bound  the expectation in  \eqref{vento1tris} by 
\be\label{samarcanda89}
\exp\{-C(  |x-x_1|  + |x_1- x_2| +\cdots + |x_{n-1}- x_n|\} \,,
\en
for some constant $C>0$ determined by $c$ and $n$.
This allows to get \eqref{vento1tris}.
For example, for $n=3$, 
we have 
\begin{equation*}
\begin{split} 
{\rm E} \left[ p_\o(x,x_1) p_\o(x_1, x_2)p_\o(x_2,x_3) \right] 
& \leq 
{\rm E} \left[ p_\o(x,x_1)\right]^\frac{1}{2} {\rm E} \left[ p_\o(x_1, x_2)p_\o(x_2,x_3) \right]^\frac{1}{2} \\& 
\leq {\rm E} \left[ p_\o(x,x_1)\right]^\frac{1}{2} {\rm E} \left[ p_\o(x_1, x_2)\right]^\frac{1}{4} {\rm E}\left[p_\o(x_2,x_3) \right]^\frac{1}{4} \\
&\leq \exp\{-\frac{c}{4}(  |x-x_1|  + |x_1- x_2| + |x_2- x_3|) \}\,.
\end{split}
\end{equation*}}

\rot{By similar arguments one can check condition \eqref{vento2}. In particular, as above, it is enough to prove that given $x\in S$ and $n \in \bbN_+$ it holds
\be\label{vento2tris}
\sum _{x_1,x_2, \cdots ,x_n\in S} {\rm E}\left[p_\o(x,x_1) p_\o(x_1, x_2) \cdots p_\o(x_{n-1}, x_n)
c^s_{x_n}(\o)\right]<+\infty\,.
\en
The only difference here is that one has to use H\"older's inequality in the  first step. In particular, one has to  bound the expectation in \eqref{vento2tris}  by 
\be\label{patroclo24}
{\rm E}\left[p_\o(x,x_1) p_\o(x_1, x_2) \cdots p_\o(x_{n-1}, x_n)\right]^\frac{\e}{1+\e}
{\rm E} \left[  c^s_{x_n}(\o)^{1+\e}\right]^\frac{1}{1+\e}<+\infty\,.
\en
Due to the stationarity of $\cP$ and our moment assumption, the expectation ${\rm E} \left[  c^s_{x_n}(\o)^{1+\e}\right]$ is bounded uniformly in $x_n$, while the first  expectation in \eqref{patroclo24}  is bounded by \eqref{samarcanda89} as already observed. This allows to conclude.}
\end{proof}
 \section{Graphical construction,  Markov generator and duality of SSEP: proofs}\label{sec_proof_SSEP}

In this section we detail the graphical construction of the SSEP, thus requiring a certain care on the measurability issues. 
At the end we provide the proof of Propositions \ref{costruzione_SSEP}, \ref{feller_SSEP}, \ref{inf_SSEP}, \rot{\ref{eastpak}}, \rot{\ref{volpino}}  and Lemma \ref{sfinimento} presented in Section \ref{sec_SSEP}.

\smallskip

Let us  start with the graphical construction.
  For later use we note 
 that $\cB( D_\bbN^{\cE_S})$ is \rosso{the $\s$--algebra} generated by  the coordinate maps $D_\bbN^{\cE_S} \ni \cK  \mapsto \cK_{x,y} \in D_\bbN$, as $\{x,y\}$ varies in $\cE_S$ (see the metric \eqref{pioggina}).  Since  $\cB(D_{\bbN})$ is generated by the coordinate maps $D_{\bbN}\ni  \xi \mapsto   \xi (t) \in \bbN$ with $t\in \bbR_+$ 
 (see Section \ref{sec_noto}), we get that $\cB( D_\bbN^{\cE_S})$  is the $\s$--algebra generated by the maps $\cK \mapsto \cK_{x,y}(t)$ as  $\{x,y\}$ varies in $\cE_S$ and $t$ varies in $\bbR_+$.

\smallskip

Given  $\cK\in D_\bbN^{\cE_S}$, $\cK_{x,y}$ is a c\`adl\`ag path with  values in $\bbN$, hence $\cK_{x,y}$ has a finite set of jump times on any interval $[0,T]$, $T>0$.
Given $x\in S$ and $\cK\in D_\bbN^{\cE_S}$, we define $\cK_x:\bbR_+\to \bbN\cup\{+\infty\}$ as 
\be\label{venzone}
\cK_x(t):=\sum_{y\in S: y\not =x } \cK_{x,y}(t)\,.
\en
 Since each map $\cK\mapsto \cK_{x,y}(t)$ is measurable, the same holds for the map $\cK \mapsto \cK_x(t)$. We point out that the path $\cK_x  $ is not necessarily c\`adl\`ag.  For example, take $\cK\in D_{\bbN}^{\cE_S}$ such that, for $1\leq i < j$, 
 \[ \cK_{s_i,s_j }(t): =
 \begin{cases}
 0 &\text{ if } i\not =1 \,,\\
 \mathds{1}( t\geq 1+1/j) &\text{ if }i=1\,.
 \end{cases}
 \]
Then $\cK _{s_1} (t) $ equals zero for $t\leq 1$ and equals $+\infty$ for $t>1$. 

\begin{Definition}\label{def_gamma}
We define $\G\subset D_{\bbN}^{\cE_S}$ as the family of  $\cK\in D_{\bbN}^{\cE_S}$ such that 
\begin{itemize}
\item[(i)]
for all $x\in S$, $K_x(t)$ is finite for all $t\in \bbR_+$ and 
  the map 
$K_x:\bbR_+\to \bbN$ is c\`adl\`ag with jumps of value $+1$;
\item[(ii)] for any $\{x,y\}\in \cE_S$ the path $\cK_{x,y}$ 
 has only  jumps of value $+1$;
\item[(iii)]  given any $\{x,y\}\not=\{x',y'\}$ in $\cE_S$, the set of  jump times of $\cK_{x,y}$ 
and the set of jump times of $\cK_{x',y'}$ 
 are disjoint;
 \item[(iv)] \rot{$\cK_{x,y}(0)=0$ for all $\{x,y\}\in \cE_S$.}
 \end{itemize}
 \end{Definition}

\begin{Lemma}\label{muset}
$\G$ is  measurable, i.e.~$\G\in \cB(D_{\bbN}^{\cE_S}) $, and $\bbP(\G)=1$. 
\end{Lemma}
\begin{proof} It is a standard fact that the $\cK$'s satisfying \rot{(ii), (iii) and (iv)} form a measurable set $\rot{\cD} \subset D_\bbN^{\cE_S}$ with   $\bbP$--probability one.  By property (ii)   for all $\cK\in \rot{\cD}$ and $x\in S$ the path $K_x$ is weakly increasing. Hence 
 $\cA:=\cap_{x\in S}\{ \cK\in\rot{\cD}\,:\, \cK_x(t) <+\infty \;\forall t \geq 0\}=\cap_{x\in S}\cap _{t\in \bbQ_+} \{\cK\in\rot{ \cD}\,:\, \cK_x(t)<+\infty \}$, where $\bbQ_+:=\bbQ\cap \bbR_+$.  Since the map $\cK\mapsto \cK_x(t)$ is measurable, also the set  $\{\cK\in\rot{ \cD}\,:\, \cK_x(t)<+\infty \}$ is measurable. Moreover this set has $\bbP$-probability $1$ since $\bbP(\rot{\cD})=1$ and, 
by (C1),  $\bbE[ \cK_x(t)]=\sum_{y\in S:y\not=x} c_{x,y}=c_x <+\infty$. This allows to conclude  that $\cA$  is measurable and  $\bbP(\cA)=1$.

To conclude it is enough to show that $\G=\cA$. Trivially $\G\subset \cA$. To prove that $\cA\subset \G$ we just need  to show that,   given $\cK\in \cA$ and  $x\in S$, $K_x:\bbR_+\to \bbN$ is c\`adl\`ag with jumps of value $+1$.    Let us  first prove that $K_x$ is  c\`adl\`ag .
 To this aim  we fix $t\geq 0$ and take $T>t$. We note that $\sum_{y \in S:y\not =x} \cK_{x,y}(T)=\cK_x(T)<+\infty$, 
$\cK_{x,y} (s) \leq \cK_{x,y} (T)$  for any $s\in (t,T)$ and $\lim_{s\downarrow t} \cK_{x,y}(s)= \cK_{x,y}(t)$ (for the first two properties use that $\cK\in \cA$, for the third one use that $\cK_{x,y}\in D_\bbN$).  Then, by dominated convergence \rosso{applied to the sum among $y\in S$ with $y\not=x$}, we get that $\lim_{ s\downarrow t} \cK_x(s)=\sum_{y\in S:y\not =x} 
\lim_{s\downarrow t}\cK_{x,y}(s)=\cK_x(t)$ (i.e.~$\cK_x$ is right-continuous). Similarly, for $t>0$, we get  that 
 $\lim_{ s\uparrow  t} \cK_x(s)=\sum_{y\in S:y\not =x} 
\lim_{s\uparrow t}\cK_{x,y}(s)=\sum_{y\in S:y\not =x} 
\cK_{x,y}(t-)$ (i.e.~$\cK_x$ has left limits).
Due to the above identities, if $t$ is a jump time of the c\`adl\`ag path $\cK_x$, then the jump value is $\cK_x(t)-\cK_x(t-)= \sum_{y\in S:y\not =x} 
\cK_{x,y}(t)-\sum_{y\in S:y\not =x} 
\cK_{x,y}(t-)$. The above expression and  properties (ii) and (iii) imply that $\cK_x(t)-\cK_x(t-)=+1$.
\end{proof}

 \begin{Definition}\label{def_jack} Let $\cK\in \G$. 
 Given $W\subset \bbR_+$
 and $\{x,y\}\in S$ we define $J_W(\cK_{x,y})$ as the set of jump times of $\cK_{x,y}$ which lie in  $W$. Similarly, given $x\in S$, we define $J_W(\cK_x) $ as the set of jump times of $\cK_{x}$ which lie in  $W$. 
 \end{Definition}
%
Note that, if $\cK\in \G$, then 
  $
   J_W(\cK_x)= \cup_{y\in S\setminus\{x\} }J_W( \cK_{x,y})
   $
   and $J_{[0,T]}   (\cK_x)$ is finite for all   $x\in S$ and $T\geq 0$. \rot{Moreover, $0$ cannot be a jump time of $K_{x,y}$ or $K_x$ since both paths $K_{x,y}$ and $K_x$ are c\`adl\`ag. In particular, $J_{(0,t]}(\cK_x)= J_{[0,t]}(\cK_x)$ and $J_{(0,t]}(\cK_{x,y})= J_{[0,t]}(\cK_{x,y})$.}

 \bigskip

We introduce $\partial$ as  an abstract state not in $S$.
Given $\cK\in \G$ and   given $t\geq 0$, we associate to each $x\in S$ an element $X^x_t[\cK]$ in 
 $S\cup\{\partial\}$ as follows:

\begin{Definition} [Definition of  $X^x_t\text{[}\cK\text{]}$ for $\cK\in \G$ and $t\geq 0$] \label{rumore} $\,\,$ \\
$\bullet$ We first consider the  set $A_1:= J_{[0,t]}(\cK_x) $. If $A_1$ is empty, then we set $X^x_t[\cK]:=x$ and we stop. If $A_1$ is nonempty, then  we define 
 $t_1$ as  the \rosso{largest}  time in $A_1$  and $x_1$ as the unique point in $S$ 
  such that $t_1 \in J_{[0,t]}(\cK_{x,x_1})$\footnote{Note that $t_1$ and $x_1$ are well defined since $\cK\in \G$.}.

$\bullet$  In general,  arrived  at $t_k,x_k$ with $k\geq 1$, we  consider the set $A_{k+1}:= J_{[0,t_k)}(\cK_{x_k})$.  
  If 
  $A_{k+1}$ is empty, then we set $X^x_t[\cK]:=x_k$ and we stop. If $A_{k+1}$ is nonempty, then  we define 
 $t_{k+1}$ as  the \rosso{largest}  time in $A_{k+1}$ and $x_{k+1}$ as the unique point in $S$ 
  such that $t_{k+1}\in J_{[0,t_k)}(\cK_{x_k,x_{k+1}})$.

$\bullet$ If the algorithm stops in a \rot{finite} number of steps, then  $X^x_t[\cK]$ has been defined  by the algorithm itself and belongs to $S$. If the algorithm does not stop, then we set  $X^x_t[\cK]:=\partial$.
\end{Definition}
\begin{Remark}\label{sirialuna}  Note that $X_0^x[\cK]=x$ according to the above algorithm \rot{(indeed, in this case $A_1=\emptyset$ as $\cK_x$ has no jump at time $0$ being c\`adl\`ag)}.
\end{Remark}


\medskip

We endow the countable set $S\cup \{ \partial\}$ with the discrete topology. 
Below, when considering functions with domain  $\G$,  we consider $\G\in \cB(D_\bbN^{\cE_S})$ as a measurable space with $\s$--algebra $\{A\cap \G\,: A\in \cB(D_\bbN^{\cE_S}\}= \{A \in \cB(D_\bbN^{\cE_S}): A\subset \G\}$.
\begin{Lemma}\label{mes0} Given $x\in S$ and $t\geq 0$, 
the map 
$\G \ni \cK \mapsto  X^x_t[\cK]\in S\cup\{\partial\}$
is measurable.
\end{Lemma}
\begin{proof} We just need to show that, given $y\in S$, the set  $\cW=\{\cK\in \G\,:\,
 X^x_t[\cK]=y\}$ is measurable.  Let us call $N_x(\cK)$ the number of steps in the algorithm of Definition  \ref{rumore}. Below $x_1,t_1,x_2,t_2,..$ are as in the above algorithm. Moreover, to simplify the notation, we take $t=1$.
 
 We first take $y\not =x$. The above set  $\cW$ 
is then the union of the countable family of sets $\{\cK\in \G \,:\, N_x(\cK) =m,\,x_1=a_1,\, x_2=a_2\,,\,...\, x_{m-1}=a_{m-1}\,,\; x_m =y\}$,  where 
$m\in \{1,2,\dots\}$,  $a_0:=x$, $a_m:=y$ and $a_1,a_2, \dots, a_{m-1} \in S$ satisfy $a_i \not =a_{i+1}$ for $i=0,2,\dots, m-1$. Let us show for example that the set
$\cA:=\{\cK\in \G \,:\, N_x(\cK) =1,\,x_1=y \}$  is measurable (the case with $N_x(\cK) =m\geq 2$ is similar, just more involved from a notational viewpoint).
We claim that  
\be\label{carattere}
\cA=\rosso{ \cup _{n_0 \geq 1}\left( \cap _{n\geq n_0} \left(\cup_{k =1}^{2^n} \cA_{n,k}\right)\right)}\,,
\en where 
\begin{equation*}
\begin{split}
\cA_{n,k}:=
\bigl\{ &\cK \in \G\,:\,   \cK_x(1)
 =\cK_x(k2^{-n})\,,\;  \cK_x(k 2^{-n})- \cK_x((k-1)2^{-n})=1\,,\; \\
&  \cK_{x,y}(k2^{-n})- \cK_{x,y}((k-1)2^{-n})=1\,,\;\cK_{y}((k-1)2^{-n})= \cK_y(0)\bigr\}\,.
\end{split}
\end{equation*}

To prove our claim take  $\cK \in \cA$.  Since $\cK\in \G$, we know that the map $\cK_x$  takes value in $\bbN$ and    is  c\`adl\`ag with jumps equal to $+1$.  Recall that $t_1>0$ is the last time in the nonempty set  $J_{[0,1]}(\cK_x)$.   This implies that $\cK_x(1)=\cK_x(s)$ for $t_1 \leq s \leq 1$.
Given $n_0$ and $n\geq n_0$ let $k\in\{1,\dots, 2^n\}$ with $(k-1)2^{-n}<t_1\leq k 2^{-n} $. Then, by the  above observation, $\cK_x(1)= \cK_x(k 2^{-n})$.
Since $\cK_x$ is c\`adl\`ag,
 there exists $\e>0$ small  such that   $\cK_x(t_1)-1=\cK_x(s)$ for all $s\in [t_1-\e, t_1)$. Therefore,  
 by taking $n_0$ such that $2^{-n_0}<\e$ and using that $\cK\in \cA$, we have that $ \cK_x(k 2^{-n})- \cK_x((k-1)2^{-n})=1=\cK_{x,y}(k 2^{-n})- \cK_{x,y}((k-1)2^{-n})=1$. By definition of $\cA$ we know that $\cK_y$ has no jumps in $[0,t_1)$, thus implying that $\cK_{y}((k-1)2^{-n})= \cK_y(0)$. This concludes the proof that, if $\cK \in \cA$, then $\cK$ belongs to  r.h.s. of \eqref{carattere}. 
  
 Suppose now that $\cK$ belongs to the r.h.s. of \eqref{carattere}. \rosso{Let us} prove that $\cK\in \cA$. First we observe that, for some $n_0\geq 1$ and for all $n\geq n_0$, we can find $k_n \in \rosso{\{1,2,3,\dots,2^n\}}$ such that $\cK\in \cA_{n,k_n}$. 
 By definition of $\cA_{ n, k_n}$ we have
  $ \cK_x(1)=\cK_x(s)$ for all $s\in [ k_n 2^{-n}, 1]$ and $ \cK_x(k_n 2^{-n})- \cK_x((k_n-1)2^{-n})=1$. To simplify the notation let $a:= (k_n-1)2^{-n}$ and $b:= k_n 2^{-n}$. By the above properties \rosso{and since $\cK\in \G$}, $\cK_x$ has exactly one jump time  in $(a,b]$. Setting $c:= (a+b)/2$,
  $a<c<b$ are subsequent points in 
    $2^{-n-1} \bbN$. If the above jump time lies in $(c,b]$, then 
  it must be $k_{n+1} 2^{-n-1}=b= k_n 2^{-n}$ and $(k_{n+1} -1)2^{-n-1}=c>a=(k_n-1)2^{-n}$. If the above jump time lies in $(a,c]$, then 
  it must be $k_{n+1} 2^{-n-1}=c< k_n 2^{-n}=b$ and $(k_{n+1} -1)2^{-n-1}=a=(k_n-1)2^{-n}$.
  This proves that the 
  sequence $k_n 2^{-n}$
 is weakly decreasing in $[0,1]$, while  the sequence $(k_n-1)2^{-n}$ is weakly increasing in $[0,1]$. Since in addition the terms $(k_n-1)2^{-n}$ and $k_n 2^{-n}$ differ by $2^{-n}$, we obtain that $(k_n-1)2^{-n}$ converges to some $u\in [0,1]$ from the left and $k_n 2^{-n}$ converges to the  same $u\in [0,1]$ from the right. It cannot be $u=0$ otherwise it would be $(k_n-1)2^{-n}=0$ for $n\geq n_0$  and we would have 
 $\cK_{x,y}(2^{-n})-\cK_{x,y}(0)=1$ (due to the definition of $\cA_{n,k_n}$), thus contradicting the right continuity of $\cK_{x,y} $.    Hence $u>0$.
We claim that $\cK\in \cA$ and $u=t_1$. By taking the limit $n\to +\infty$ in the properties defining  $\cA_{n,k_n}$ and using that $\cK_{x,y} $, $\cK_x $ and $\cK_y$ are c\`adl\`ag, we get
$\cK_x(1)= \cK_x(u)$,
$ \cK_x(u)- \cK_x(u-)=1$,
$ \cK_{x,y}(u)- \cK_{x,y}(u-)=1$, $ \cK_{y}(u-)= \cK_y(0)$. This  implies that $\cK\in \cA$ and $u=t_1$.
 At this point we have proved our claim \eqref{carattere}. 

Having \eqref{carattere}, which involves countable intersections and unions, the measurability of $\cA$ follows from the measurability of $\cA_{n,k}$ (for the latter recall that the  maps $D_\bbN^S \ni \cK \to \cK_{x,y}(s)\in \bbN$   and  $D_\bbN^S \ni \cK \to \cK_{x}(s)\in \bbN$ are measurable).

When $y=x$, the analysis is the same. One has just to consider in addition the event $\{\cK\in \G\,:\, N_x(t)=0\}=\{\cK\in \G\,:\,\cK_x(0)=\cK_x(t)\}$, which is trivially measurable.
\end{proof}

\begin{Lemma}\label{peratoons}
Given $\cK \in \G$ and $x\in S$, the path $\bbR_+\ni t \mapsto X^x_t[\cK]\in S\cup\{\partial\}$ is c\`adl\`ag.
\end{Lemma}
\begin{proof}  We prove the c\`adl\`ag property at $t>0$, the case $t=0$ is similar. We fix  $\cK \in \G$, $x\in S$.
  Since $\cK\in \G$, $\cK_x$ is c\`adl\`ag and therefore $\cK_x$ has no jump in $(t,t+\e]$ and in $[t-\e,t)$ for $\e>0$ small enough.
  By the first step in the algorithm in Definition \ref{rumore} we conclude respectively  that $X^x_s[\cK]=X^x_t[\cK]$  for any $s\in [t,t+\e]$ and $X^x_{s}[\cK]=\rosso{X^x_{t-\e}[\cK]}$ for for any $s\in [t-\e,t)$.  
\end{proof}

 Let us point out an important  symmetry (called below \emph{reflection invariance}) of the law $\bbP$. 
 We define $\t_{x,y}^{(1)}<\t_{x,y}^{(2)}< \cdots $ the jump times of the random path $N_{x,y}(\cdot )$. Setting $\t_{x,y}^{(0)}
:=0$, we have that $\t_{x,y}^{(k)}- \t_{x,y}^{(k-1)}$, $k\geq 1$, are i.i.d. exponential random variables of parameter $c_{x,y}$. Another characterization is that $\{ \t_{x,y}^{(k)}\,:\, k \geq 1\}$ is a Poisson point process with intensity $c_{x,y}$. As a consequence, given $T>0$, the set $(0,T)\cap \{ \t_{x,y}^{(k)}\,:\, k \geq 1\}$ is a Poisson point process on $(0,T)$ with intensity $c_{x,y}$, i.e.~the numbers of points in disjoint Borel subsets of $(0,T)$ are independent random variables and the number of points in a given Borel subset $A$ of $(0,T)$ is a Poisson random variable with parameter $c_{x,y} $ times   the Lebesgue measure of $A$.   Due to the above characterization of the Poisson point process, we have that also $(0,T)\cap \{T- \t_{x,y}^{(k)}\,:\, k \geq 1\}$ is a Poisson point process on $(0,T)$ with intensity $c_{x,y}$.

Due to
the above reflection invariance,  the independence of the Poisson processes and the algorithm in Definition \ref{rumore},  we get that, by sampling $\cK$ with probability $\bbP$, the random variable  $ X^x_t[\cK]$ is distributed as 
the state at time $t$  of the random walk with conductances $c_{y, z}$  starting at $x$, with the convention that $ X^x_t[\cK]=\partial$ corresponds to the explosion of the above random walk.

The  above  observation will be crucial in proving Lemma \ref{mes1} below. We first give the following definition:

\begin{Definition}\label{zucchine} We define the set $\G_*\subset \G$ as $ \G_*:= \{\cK\in \G\,:\, X^x_t[\cK]\in S \text{ for all }x\in S, t\geq 0\}$.
\end{Definition}

\begin{Lemma}\label{mes1} The set $\G_*$ 
 is measurable  and $\bbP(\G_*)=1$.
  \end{Lemma}
 \begin{proof} Let $ \G_{\circ}:= \{\cK\in \G\,:\, X^x_t[\cK]\in S \text{ for all }x\in S, t\in \bbQ\cap \bbR_+\}$. By Lemma \ref{peratoons}, 
  $\G_{\circ}=\G_*$. 
 Due to Lemma \ref{mes0} the set $\G_t:=\{ \cK\in \G\,:\, X^x_t[\cK]\in S\;\; \forall x\in S\}$ is measurable.
    Due to  the above interpretation of the  distribution of $ X^x_t[\cK]$ in terms of the random walk with random conductances  and due to  Condition (C2),  we have   $\bbP(\G_t )=1$.   Then, since
 $\G_{\circ}$   is the countable intersection of the  measurable sets $\G_t$, $t\in \bbR_+\cap \bbQ$, we conclude that $\G_\circ$ is measurable and $\bbP(\G_\circ)=1$.  As $\G_{\circ}=\G_*$, the same holds for $\G_*$.
  \end{proof}
 
 \begin{Definition}\label{croato}
  For each $\cK\in \G_*$, $t\geq 0$ and $\s\in \{0,1\}^S$, we define $\eta^\s _t[\cK]\in \{0,1\}^S$ as 
 \be\label{def_fond}
\eta^\s_t [\cK] (x):= \s\bigl( X^x_t[\cK]\bigr) \,.
 \en
 \end{Definition}
 
 \rot{We note that, given $\cK\in \G_*$, it holds  $\eta^\s_0 [\cK] =\s$ by Remark \ref{sirialuna}. We refer to Figure \ref{fig2} for an example illustrating our definitions. }

\begin{figure}
\includegraphics[scale=0.18]{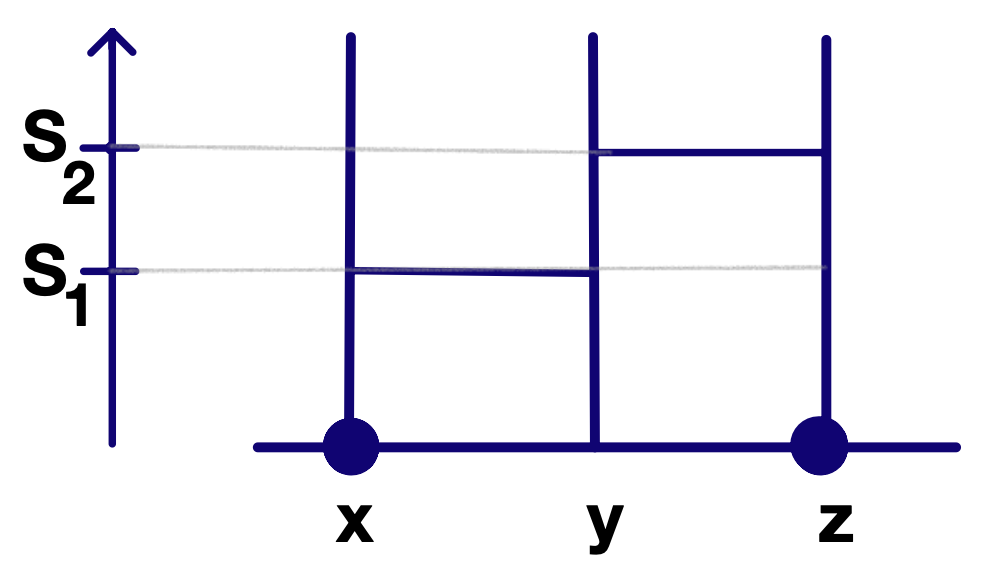}
\begin{tabular}{|c|c|c|c|}
\hline 
 &   $ 0\leq t <s_1$ & $s_1\leq t <s_2$ & $t=s_2$\\
\hline
$X^x_t[\cK]$ & $x$ & $y$ & $y$\\
\hline
$X^y_t[\cK]$ & $y$ & $x$ & $z$\\
\hline
$X^z_t[\cK]$ & $z$ & $z$ & $x$\\
\hline
$\eta^\s_t[\cK]$ & $(.., 1,0,1,..)$ & $(..,0,1,1,..) $ & $(..,0,1,1,..)$\\
\hline
 \end{tabular}
\caption{(Top) Horizontal edges represent the jump times of $\cK_{x,y}$ and $\cK_{y,z}$. (Bottom) Value of 
$X^x_t[\cK]$, $X^y_t[\cK]$, $X^z_t[\cK]$ and $\eta^\s_t[\cK]$  for $t\in [0,s_2]$.
} \label{fig2}
\end{figure}

%
 \begin{Lemma}\label{messi} Fixed $t\geq 0$, the map $ \{0,1\}^S \times \G_* \ni (\s,\cK) \mapsto  \eta^\s_t [\cK]  \in \{0, 1\}^S$ is continuous in $\s$ (for fixed $\cK$) and measurable in $\cK$ (for fixed $\s$).  
 \end{Lemma}
 \begin{proof} We prove the continuity in $\s$ for fixed $\cK$. Recall the metric $d(\cdot, \cdot)$ on $\{0,1\}^S$ defined in  \eqref{sarto}. 
  Fix $N\in \bbN_+$. We have that $d( \eta^\s_t [\cK], \eta^{\s'}_t [\cK])\leq 2^{-N}$ if $\eta^\s_t [\cK](s_n)= \eta^{\s'}_t [\cK](s_n)$ for all $n=1,2,\dots, N$. By \eqref{def_fond} this holds if $\s(y)=\s'(y)$ for all $y \in A:= \{  X^{s_n}_t[\cK]\,:\, n=1,2,\dots, N\}$. Let us now choose $M$ large enough that $A\subset \{s_1,s_2,\dots, s_M\}$.
  Since we have  $\s(s_n)=\s'(s_n)$  for all $n=1,2,\dots,M$ if $d(\s,\s')< 2^{-M}$, we conclude that whenever 
  $d(\s,\s')< 2^{-M}$ we have $\s(y)=\s'(y)$ for all $y \in A$ and therefore $d( \eta^\s_t [\cK], \eta^{\s'}_t [\cK])\leq 2^{-N}$. This proves the continuity in $\s$.
  
 We prove the measurability in $\cK$ for fixed $\s$.  The Borel $\s$--algebra $\cB(\{0,1\}^S)$ is generated by the sets $\{\eta\in \{0,1\}^S\,:\, \eta(x) =1\}$ as $x$ varies in $S$. Then to prove the measurability in $\cK$, for fixed $\s$, we just need to show that  $B:=\{\cK \in \G_*\,:\, \eta_t^\s[\cK](x) =1\}$ is measurable for any $x\in S$.  By \eqref{def_fond} we can rewrite $B$ as the countable union $\cup_{y\in \bbZ^d  \,:\s(y)=1} \{ \cK\in \G_*\,:\, X^x_t[\cK]=y\}$. Then the measurability of $B$ follows from the measurability of the sets $\{ \cK\in \G_*\,:\, X^x_t[\cK]=y\}$, which follows from   Lemmas \ref{mes0} and \ref{mes1}.
 \end{proof}
 For the next result recall that $D_{\{0, 1 \}^S}  = D(\bbR_+,\{0, 1 \}^S)$.
 \begin{Lemma}\label{pablito} For each $\s\in \{0,1\}$ and $\cK\in \G_*$, the path 
 $\eta^\s _\cdot[\cK] :=\bigl(  \eta^\s_t [\cK] \bigr)_{t\geq 0}$ belongs to $D_{\{0, 1 \}^S}$. Moreover, fixed $\s \in \{0,1\}^S$, the map
 \be\label{gallavotti}
  \G_* \ni \cK \mapsto  \eta^\s_\cdot [\cK]  \in D_{\{0,1\}^S} 
  \en
  is measurable in $\cK$.
 \end{Lemma}
 \begin{proof}  Let $\s\in \{0,1\}^S$ and $\cK\in \G_*$.
 We first check that the path 
 $\bigl(  \eta^\s_t [\cK] \bigr)_{t\geq 0}$ belongs to $D_{ \{0, 1 \}^S}$.   
 We first prove its right continuity, i.e.~given $t\geq 0$ we show that $\lim_{u\downarrow t} \eta^\s_u [\cK]=\eta^\s_t [\cK]$. Due to \eqref{sarto} we just need to show that  $\lim_{u\downarrow t} \eta^\s_u [\cK](x)=\eta^\s_t [\cK](x)$ for any $x\in S$.  By \eqref{def_fond} it is enough to show that, for any $x\in S$,  $X^x_u[\cK]= X^x_t[\cK]$ for any $u>t$ sufficiently near to $t$. This follows from Lemma \ref{peratoons}. 
 We now  prove that the path  $\eta^\s _\cdot[\cK] $ has limit from the left at $t>0$. 
 Given $x\in S$,  
  by Lemma \ref{peratoons}, the left limit $X^x_{t-}[\cK]$ is well defined. Moreover, since $\cK\in \G_*$, this limit is in $S$. Let $\xi(x):= \s( X^x_{t-}[\cK])\in \{0,1\}$ for any $x\in S$. We claim that $\lim_{u\uparrow t} \eta^\s_u [\cK]=\xi$. Again, due to \eqref{sarto}, we just need to show that  $\lim_{u\uparrow t} \eta^\s_u [\cK](x)=\xi(x)$, i.e.
  $\lim_{u\uparrow t} \s( X^x_u [\cK])=\xi(x)= \s( X^x_{t-}[\cK])$,
     for any $x\in S$.  This follows from the fact that, for each $x\in S$, $X^x_u [\cK]=
     X^x_{t-} [\cK]$ for all $u\in [t-\e,t)$, for $\e>0$ small enough. This concludes the proof that 
     $\bigl(  \eta^\s_t [\cK] \bigr)_{t\geq 0}$ belongs to $D_{ \{0, 1 \}^S}$.  
     
 
 As  discussed at the end of Section \ref{sec_noto}, $\cB( D_{\{0,1\}^S})$ is generated by the maps $\eta_\cdot \mapsto  \eta_t $ with $t$ varying in $\bbR_+$. This means that $\cB( D_{\{0,1\}^S})$  is the smallest $\s$--algebra such that the sets of the form  $\cU:=\{ \eta_\cdot \in D_{\{0,1\}^S}\,:\, \eta_t \in B\}$, with $B\in \cB (\{0,1\}^S)$, are measurable. To prove the measurability of \eqref{gallavotti}   in $\cK$ for fixed $\s$, we just need to prove that the inverse image of $\cU$ via \eqref{gallavotti} is  measurable, i.e.~that the set  $\{\cK\in \G_*\,:\,  \eta^\s_t [\cK]  \in B\}$ is measurable. But this is exactly the measurability in $\cK$ (for fixed $\s$) of the map  in Lemma \ref{messi}.
 \end{proof}


\subsection{Proof of Proposition \ref{costruzione_SSEP}}\label{avatar} $\,$

$\bullet$ Item  (i)  \rot{holds since, as already observed, given $\cK\in \G_*$ it holds  $\eta^\s_0 [\cK] =\s$ by Remark \ref{sirialuna}. } 
 

$\bullet$  We move to Item (ii). We will use the Dynkin's $\pi$-$\l$ Theorem (see e.g.~\cite{Du})     as in the proof of Item (b) in   \cite[Theorem~2.4]{timo}. 
 We fix $0 \leq t_1\leq t_2 \leq \cdots \leq t_n$, $x_1,x_2,\dots, x_n \in S$ and $k_1,k_2,\dots, k_n \in \{0,1\}$. Due to Lemma \ref{messi}, for fixed $\cK\in  \G_*$, the maps 
 $\{0,1\}^S \ni \s\mapsto  \eta^\s_{t_i}[\cK](x_i)\in \{0,1\}$ are continuous.
Therefore, \rosso{also} 
\[
f_\s(\cK):= \prod _{i=1}^n \left(1-\big|  \eta^\s_{t_i}[\cK](x_i)-k_i\big| \right)
\] is continuous in $\s$ for each fixed $\cK$. Trivially 
$f_\s(\cK)=1$ if and only if $ \eta^\s_{t_i}[\cK](x_i)=k_i$ for all $i=1,2,\dots,n$, otherwise $f_\s(\cK)=0$.
  We then consider the set  
  $A:=\{\, \eta_\cdot \in D_{ \{0,1\}^S}\,:\, \eta_{t_i}(x_i)=k_i \; \forall i=1,2,\dots,n\,\}$. Note that $A\in \cF$ as $\cF$ is generated by the coordinate maps, \rosso{moreover $f_\s(\cK)$ is measurable in $\cK$ by Lemma \ref{messi}}.   
 Since $\bbP^\s(A)= \int d\bbP(\cK) f_\s(\cK)$ and $0 \leq f_\s(\cK) \leq 1$, by dominated convergence and the above continuity in $\s$ of $f_\s(\cK)$, we get that the map $\{0,1\}^S\ni \s\mapsto \bbP^\s( A) \in [0,1]$ is continuous.

 
  We now consider the family $\cL$  of sets $A\in \cF$ such that the map $\{0,1\}^S\ni \s\mapsto \bbP^\s( A) \in [0,1]$   is measurable. Due to the above discussion, $\cL$ contains the family $\cP$ given by events of the form $\{\, \eta_\cdot \in D_{ \{0,1\}^S}\,:\, \eta_{t_i}(x_i)=k_i \; \forall i=1,2,\dots,n\,\}$ with $n\geq 1 $, $0 \leq t_1\leq t_2 \leq \cdots \leq t_n$, $x_1,x_2,\dots, x_n \in S$ and $k_1,k_2,\dots, k_n \in \{0,1\}$. 
    The above family $\cP$ is a $\pi$--system, i.e.~$A\cap B \in \cP$ if $A,B\in \cP$. We claim that, on the other hand, $\cL$ is a $\l$--system, i.e.~(a) $D_{\{0,1\}^S}\in \cL$, (b) if $A,B\in \cL$ and $A\subset B$, then $B\setminus A\in \cL$, (c) if $A_n\in \cL$ and $A_n \uparrow A$ (i.e.~$A_1\subset A_2\subset \cdots$ and $\cup_n A_n=A$), then $A\in \cL$. Before justifying our claim, let us conclude the proof of Item (ii). By Dynkin's $\pi$-$\l$ Theorem, $\cL$ contains the $\s$--algebra generated by $\cP$, which is indeed $\cF$ as discussed at the end of Section \ref{sec_noto}. Hence $\cL=\cF$, thus corresponding to Item (ii).
  
  We conclude by proving our claim.
The check of (a) and (b) is trivial. Let us focus on (c). Let $A_n\in \cL$ with $A_n \uparrow A$. We need to prove that $A\in \cL$.  We first observe that for each path $\eta_\cdot \in D_{ \{0,1\}^S}$ it holds $ \mathds{1}_{A_n}( \eta_\cdot)\to \mathds{1}_A(\eta_\cdot)$ as $n\to +\infty$  \rot{by definition of $A_n \uparrow A$}. This implies that, for each $\s\in \{0,1\}^S$ and $\cK\in \G_*$,  $g_n(\cK) \to g(\cK)$ as $n\to +\infty$, where 
$g_n(\cK):= \mathds{1}_{A_n} \big( 
\bigl( \eta^\s_t [\cK] \bigr)_{t\geq 0} \big)$ and $g(\cK):= \mathds{1}_{A} \big( 
\bigl( \eta^\s_t [\cK] \bigr)_{t\geq 0} \big)$. Then, by  dominated convergence, $\bbP^\s(A_n)= \int d \bbP( \cK) g_n(\cK)  \to\int d \bbP( \cK) g(\cK)=   \bbP^\s(A)$ as $n\to +\infty$, for any $\s\in\{0,1\}$. Since $A_n\in \cL$ the map $\{0,1\}^S\ni \s\mapsto \bbP^\s( A_n) \in [0,1]$ is measurable. As a byproduct \rot{of} the above limit,  we get that the map $\{0,1\}^S\ni \s\mapsto \bbP^\s(A)\in [0,1]$ is the pointwise limit of measurable functions, and therefore it is measurable. This concludes the proof of (c).

$\bullet$
 We now focus on Item (iii).  The proof is very close to the one of Item  (c) in   \cite[Theorem~2.4]{timo}, \rosso{although the graphical construction is different}.  Given $\cK\in D_{\bbN}^{\cE_S}$ and $t\geq 0$, we define $\theta_t \cK$ as the element of $D_{\bbN}^{\cE_S}$ such that $(\theta _t \cK)_{x,y} (s):= \cK_{x,y} (t+s)-\cK_{x,y}(\rot{t})$. \rot{We note that $\theta_t \cK\in \G_*$ for any $\cK\in \G_*$}. We also denote by $\cK_{[0,t]}$ the collection of functions $[0,t]\ni s \mapsto  \cK_{x,y}(s)\in \bbN$  as $\{x,y\}$ varies in $\cE_S$. We claim that for all $ \cK \in \G_*$ and $t,s \geq 0$  it holds
 \be\label{andreino1}
 \eta ^\s_{t+s}[\cK]= \eta ^\xi _s [\theta _t \cK] \qquad \xi:= \eta ^\s_t [\cK]\,.
 \en
 To check the above identity, observe that by the graphical construction $X_{t+s}^x[\cK]=X_t^{y}[ \cK]$ where $ y:=X_s^x [\theta_t \cK]$, and therefore (defining $\xi$ as in  \eqref{andreino1}) 
 \[
  \eta ^\s_{t+s}[\cK](x)= \s( X_{t+s}^x[\cK])= \s( X_t^{y}[ \cK])= \eta^\s_t[\cK](y)=\xi (y)= \xi (X_s^x [\theta_t \cK])= \eta ^\xi _s [\theta _t \cK] \,.
  \]
 Take now $A\in \cF$ and $B\in \cF_t$. We can think of $B$ as a subset of $D([0,t], \{0,1\}^S)$. We set $\eta_{[0,t]}:= (\eta_s)_{0\leq s\leq t}$. Then, using \eqref{andreino1},  
 \be
 \begin{split}
 & \bbP^\s\left( \eta_{t+\cdot} \in A, \eta _{[0,t]}\in B\right)=\int _{\G_*}d\bbP (\cK)\mathds{1}_A\left( \eta^\s _{t+\cdot} [\cK]\right) \mathds{1} _B\left( \eta ^\s_{[0,t]}[\cK]\right) 
 \\
 &= \int _{\G_*} d\bbP (\cK)\mathds{1}_A\left( \eta^{\eta^\s_t[\cK]} _{\cdot} [\theta_t \cK]\right) \mathds{1} _B\left( \eta ^\s_{[0,t]}[\cK]\right) \,.
  \end{split}
 \en
 Since $\mathds{1}_B\bigl( \eta^\s_{[0,t]}[\cK]\bigr) $ and $\eta^\s_t[\cK]$ depend only on  $\cK_{[0,t]}$, which is independent under $\bbP$ from $\theta_t \cK$ and since $\theta_t \cK$ and $\cK$ have the same law under $\bbP$, we have  
 \begin{multline*}
  \int _{\G_*} d\bbP (\cK)\mathds{1}_A\left( \eta^{\eta^\s_t[\cK]} _{\cdot} [\theta_t \cK]\right) \mathds{1} _B\left( \eta ^\s_{[0,t]}[\cK]\right) 
\\
=
 \int _{\G_*} d\bbP (\cK)\int _{\G_*} d\bbP(\cK') \mathds{1}_A\left( \eta^{\eta^\s_t[\cK]} _{\cdot} [\cK']\right) \mathds{1} _B\left( \eta ^\s_{[0,t]}[\cK]\right)   
 =\bbE^\s\left[ \mathds{1} _B\left( \eta _{[0,t]}\right)\rosso{ \bbP^{ \eta_t}(A) } \right]\,.
 \end{multline*} 
 By collecting our observations we get 
 \[\bbP^\s\left( \eta_{t+\cdot} \in A, \eta _{[0,t]}\in B\right)=\bbE^\s\left[ \mathds{1} _B\left( \eta _{[0,t]}\right) \rosso{\bbP^{ \eta_t}(A)} \right] \qquad \forall A\in \cF\,, \; \forall B\in \cF_t\,.
 \]
 The above family of identities leads to  Item (iii).

\subsection{Proof of Proposition  \ref{feller_SSEP}}\label{linus}
Since $\{0,1\}^S$ is compact, $f$ is uniformly bounded. Fixed $\s\in \{0,1\}^S$ the map 
$F_\s:\,\G_*\ni \cK \mapsto  f\big( \eta^\s_t [\cK] \big) \in \bbR$ is measurable (as composition of measurable functions, see Lemma \ref{messi}) and  is bounded in modulus by $\|f\|_\infty$. Fixed $\xi $ and a sequence $(\xi_n)$ in  $ \{0,1\}^S$ with $\xi_n \to \xi$, due to the continuity in $\s$ of $\eta^\s_t[\cK]$ (see Lemma \ref{messi}) and due to the continuity of $f$,  we have  $F_{\xi_n} (\cK) \to F_\xi (\cK)$ as $n\to+\infty$  for all $\cK\in \G_*$. Then, by dominated convergence, we get that $\big( S (t) f\big) (\xi_n)=\bbE[ F_{\xi_n}]\to \bbE[F_\xi]=\big( S(t) f\big) (\xi)$ as $n\to +\infty$.

\subsection{Proof of Proposition \ref{inf_SSEP}}\label{manfred}
  Let $f : \{0,1\}^S \to \bbR$ be a local function and take $A\subset S$ finite such that $f(\eta)$ is defined in terms only of $\eta(x)$ with $x\in A$.  We set $\cE(A):=\{\{x,y\} \in S\,:\, \{x,y\} \cap A \not = \emptyset\}$. By (C1) we have 
   \be\label{stimetta}c_A  := \sum _{\{x,y\} \in \cE (A)} c_{x,y} \leq 
 \sum_{x\in A} \sum _{y \in S} c_{x,y} = 
 \sum _{x\in A}c_x <+\infty\,.
 \en
 From the above bound and since $\|f\|_\infty<+\infty$ it is simple to prove that 
  the r.h.s.'s of \eqref{mammaE} and \eqref{mahmood} are absolutely convergent series in $C(\{0,1\}^S)$ defining the same function, that we denote by $\hat \cL  f$. Hence we just need to prove that $\cL f= \hat \cL f$.
   
\rot{From now on $\cK$ is assumed to be in $\Gamma_*$ (see Definition \ref{zucchine}), which has $\bbP$-probability one.}
 Since $\cK_A(t):= \sum_{\{x,y\} \in \cE(A)} \cK_{x,y}(t)$, with $\cK$ sampled by $\bbP$,  is a Poisson random variable with finite  parameter  $c_A$, it holds 
  \be\label{duino}
  \bbP ( \cK_A(t) \geq 2) = 1- e^{-c_A t} (1+ c_A  t) \leq C t^2\,.
  \en

Recall that $S=\{s_1,\,s_2, \dots\}$. We introduce on $S$ the total order $\preceq$ such that $s_i \preceq s_j$ if and only if $i\leq j$.
 When  $\cK_A(t)=1$, we define  the pair  $\{x_0,y_0\}$ as the only edge in $\cE_A$ such that $\cK_{x_0,y_0}(t)=1$, with the rule that we call  $x_0$ the minimal (w.r.t. $\preceq$) point in  $\{x_0,y_0\}\cap A$. This rule is introduced in order 
  to have a univocally defined labelling of the points in the above edge.
  
  Recall \eqref{venzone}. We observe that   
  $\cK_A(t)=1$ implies that $\cK_e(t) =0 $ for all $e\in \cE(A)$ with $e\not =\{x_0,y_0\}$, $ \cK_{x_0,y_0}(t)=1$, $\cK_{x_0}(t)=1$, $\cK_{x}(t) =0$ for all $x\in A\setminus\{x_0,y_0\}$ and  $\cK_{y_0}(t) \geq 1$.  $\cK_A(t)=1$   implies also that $\cK_{y_0}(t)=1$ if $y_0\in A$.
 Let $H :=\{ \cK_A(t)=1\}\cap \{ \cK_{y_0}(t)=1\}$.   By the above observations and   the graphical construction in Definition \ref{rumore},   $H$ also implies  that 
 $X^x_t[\cK]=x$ for $x\in A\setminus \{x_0,y_0\}$, $X^{x_0}_t[\cK]=y_0$ and 
 $X^{y_0}_t[\cK]=x_0$.   Hence, $\eta_t ^\s[\cK]=\s^{x_0,y_0}$ on $A$  when $H$ occurs.
 
As already observed, if $\cK_{A}(t)=1$ and $y_0\in A$,  then 
  $H$ must occur.
 Hence 
 \be\label{polo}
 \{ \cK_{A}(t)=1\}\setminus H=\{ \cK_A(t)=1,\; \cK_{y_0}(t) >1,\;y_0\not \in A\}\,.
 \en
 We claim that, as $t\downarrow 0$,
 \be\label{circolo}
 \bbP(\{ \cK_{A}(t)=1\}\setminus H)=o(t)\,.
 \en To prove our claim
 we estimate
 \be\label{mare}
 \begin{split}
& \bbP(
 \cK_A(t)=1,\; \cK_{y_0}(t) >1,\;y_0\not \in A)
 \\
 &  \leq \sum_{x\in A} \sum _{y \in S\setminus A}  \bbP (  \cK_{x,y}(t)=1\,, \sum_{z\in  S \setminus (A\cup\{y\})} \cK_{y,z}(t) \geq 1) \\
& \leq  t \sum_{x\in A} \sum _{y \in S } c_{x,y}   e^{- c_{x,y} t}  (1- e^{-c_y  t})\,.
\end{split}
 \en
Note that in the last  bound we have used that $\sum_{z\in  S \setminus (A\cup\{y\})} \cK_{y,z}(t)$ is a Poisson random variable with parameter $\sum_{z\in  S \setminus (A\cup\{y\})} c_{y,z}\leq c_y$, which is independent from $\cK_{x,y}(t)$.  
  By \eqref{stimetta} and the  dominated convergence theorem  applied to the r.h.s. in \eqref{mare},  we get that the r.h.s.  of \eqref{mare} is $o(t)$. \rot{By combining this result with \eqref{polo}, we get \eqref{circolo}}.

Using that  $\|f\|_\infty <+\infty$, $f( \eta^\s _t [ \cK])= f(\s)$ when $\cK_A(t)=0$, \eqref{duino} and \eqref{circolo},  we can write
\be\label{beggins}
\begin{split}
& S(t) f (\s) -f(\s) = 
\bbE\left [ \left( f( \eta^\s _t [ \cK])- f(\s)\right) \mathds{1}_H\right]+ o(t)\\
& =\sum_{\{x,y\}\in \cE(A)} \left( f(\s^{x,y})-f(\s) \right) \bbP( H, \{x_0,y_0\}= \{x,y\} ) + o(t)\,.
\end{split}
\en
Above, to simplify the notation, for the intersection of events we used the comma instead of the symbol $\cap$  (we keep the same convention also below).

If   in \eqref{beggins} we replace $ \bbP( H , \{x_0,y_0\}= \{x,y\} ) $ by $\bbP( \cK_A(t)=1, \{x_0,y_0\}= \{x,y\} )$, the global error is of order   $ o(t)$. Indeed, since the first  event is included in the second one, we have 
\begin{equation*}
\begin{split}
& \sum_{\{x,y\}\in \cE(A)}  \big| \bbP( \cK_A(t)=1, \{x_0,y_0\}= \{x,y\} )- \bbP( H , \{x_0,y_0\}= \{x,y\} )\big| \\
& = \sum_{\{x,y\}\in \cE(A)} \bbP( \{ \cK_A(t)=1\} \setminus H, \{x_0,y_0\}= \{x,y\} )=\bbP( \{ \cK_A(t)=1\} \setminus H )\,,
\end{split}
\end{equation*}
and the last expression is of order $o(t)$ due to \eqref{circolo}.
By making the above replacement we have 
\be\label{beggins2}
\begin{split}
& S(t) f (\s) -f(\s) \\
&=\sum_{\{x,y\}\in \cE(A)} \left( f(\s^{x,y})-f(\s) \right)  \bbP( \cK_A(t)=1, \{x_0,y_0\}= \{x,y\} ) + o(t)\\
&=t \sum_{\{x,y\}\in \cE(A)} \left( f(\s^{x,y})-f(\s) \right)   c_{x,y} e^{-c_{x,y} t}  \exp\Big\{- t \sum _{ 
\substack{   \{x',y'\}\in \cE(A):\\ \{x',y'\}\not= \{x,y\}}
} c_{x',y'} \Big\}+ o(t)\\
&= t \sum_{\{x,y\}\in \cE(A)} \left( f(\s^{x,y})-f(\s) \right) c_{x,y} e^{-c_A t} +o(t)
 \,.
\end{split}
\en
We apply  the dominated convergence theorem to the measure  on $\cE(A)$ giving weight $c_{x,y}$ to $\{x,y\}$  and to the $t$--parametrized functions 
$ F_t(\{x,y\}):=\left( f(\s^{x,y})-f(\s) \right) \left[e^{- c_A t}-1\right]$.
The above functions are dominated by the constant function $2\|f\|_\infty$, which  is integrable as $\sum_{\{x,y\}\in \cE(A)}   c_{x,y}=c_A<+\infty$. By dominated convergence we conclude  
that  $\lim_{t\downarrow 0} \sum_{\{x,y\}\in \cE(A)} c_{x,y} F_t(\{x,y\})=0$. By combining this observation with \eqref{beggins2} we conclude that 
\begin{equation*}
\begin{split}
S(t) f (\s) -f(\s)& =t \sum_{\{x,y\}\in \cE(A)} \left( f(\s^{x,y})-f(\s) \right)   c_{x,y} +  o(t)\\
& =t \sum_{\{x,y\}\in \cE_S} \left( f(\s^{x,y})-f(\s) \right)   c_{x,y} +  o(t) = t \hat \cL f (\s) 
+ o(t)
\,.
\end{split}
\end{equation*}
The above expression implies that $f\in \cD(\cL)$ and that $\cL f=\hat \cL f$.
\subsection{Proof of Proposition \ref{eastpak}}\label{dim_eastpak}
 \rot{We first point out  a property of $\vertiii{\cdot}$ frequently used also in the rest:
 if  $h\in C(\{0,1\})^S$ and  $\eta, \xi \in \{0,1\}^S$ satisfy $\eta(x)=\xi(x)$ for all $x\in A$  for some  $A\subset S$ (possibly  $A=\emptyset$), then  it holds 
 \be\label{eccolo8}
 | h(\eta) -h(\xi)| \leq \sum_{x\in S\setminus A} \D_h(x)\,.
 \en
 The proof of \eqref{eccolo8} is standard.  Indeed, define $\s^{(n)} \in \{0,1\}^S$, $n\in \bbN$,  by setting $\s^{(n)}(s_k)=\eta(s_k)$ if $k\leq  n$ and $\s^{(n)}(s_k)=\xi(s_k)$ if $k> n$ (recall that $S=\{s_1,s_2,\dots\}$). Then $\s^{(0)}=\xi$
   and $ \s^{(n)} \to \eta$ as $n\to +\infty$, thus implying that $h(\eta)=\lim_{n\to +\infty } h(\s^{(n)})$ due to the continuity of $h$ and therefore $| h(\eta) -h(\xi)| \leq \sum_{k=0}^\infty | h(\s^{(k+1)})-h(\s^{(k)})| $. To conclude the proof of \eqref{eccolo8} it is enough to observe that $\s^{(k+1)}$ and $\s^{(k)}$ differ at most at  $s_{k+1}$ (thus implying that $| h(\s^{(k+1)})-h(\s^{(k)})| \leq \D_h (s_{k+1})$) and are equal if $s_{k+1}\in A$.}

\rot{Let now $f$ be as in Proposition \ref{eastpak}.    We set $T_{x,y} f (\eta):=f(\eta^{x,y})-f(\eta)$. 
  Since  $|f(\eta^{x,y})-f(\eta)|\leq \D_f(x)+\D_f(y)$ by \eqref{eccolo8}, we get $\| T_{x,y} f \|_\infty \leq \D_f(x)+\D_f(y)$.
  Then,  since $\vertiii{f}_*<+\infty$, it holds 
  \be\label{bici88} \sum_{\{x,y\}\in \cE_S} c_{x,y} \|T_{x,y} f\|_\infty \leq \sum_{\{x,y\}\in \cE_S} c_{x,y}(\D_f(x)+\D_f(y))=\sum_{x\in S} c_x \D_f(x)<+\infty\,.\en
  This proves the absolute convergence of  $ \sum _{\{x,y\}\in \cE_S} c_{x,y} \bigl[ f(\eta^{x,y})-f(\eta) \bigr]$ as series of functions in $C(\{0,1\}^S)$. Let us call $h(\eta)$ the resulting function. }

 \rot{Since $\cL$ is a Markov generator,  its graph is  closed in $C(\{0,1\}^S)\times C(\{0,1\}^S)$  (cf.~\cite[Chapter~1]{L1}). Hence, to conclude, it is enough to exhibit a sequence of local functions $f_n$ such that $\| f_n -f\|_\infty\to 0$ and $\|  \cL f_n - h\|_\infty \to 0$ as $n\to +\infty$ (indeed this implies that $f\in \cD(\cL)$ and $h=\cL f$). To this aim, given $\eta \in \{0,1\}^S$, we define    $\eta^{(n)}(x)=\eta(x)$ if $x\in A_n:=\{s_1,s_2, \dots, s_n\}$ and  $\eta^{(n)}(x)=0$ otherwise. We also set $f_n(\eta):= f( \eta^{(n)})$. Trivially $f_n$ is a local function. Moreover,  by \eqref{eccolo8},
 $\| f- f_n\|_\infty  \leq \sum _{x\not \in A_n}  \D_f (x) $, which goes to zero as $n\to +\infty$ since $\vertiii{f}<+\infty$.  It remains to  prove  that 
 $\|  \cL f_n - h\|_\infty \to 0$.  To this aim we fix $\e>0$. Due to \eqref{bici88} and since $\vertiii{f}_*<+\infty$, we can fix $N$ such that
 \be \label{ululato1}
 \sum_{\{x,y\} \in \cE_S:\\
 \{x,y\}\not \subset A_{N} } c_{x,y} \|T_{x,y} f\|_\infty 
 \leq \e\,.
 \en 
 Note that $\| T_{x,y} f_n \|_\infty \leq \| T_{x,y} f \|_\infty$, hence \eqref{ululato1} holds also with $f_n$ instead of $f$ for any $n=1,2,\dots$.
 As a consequence and due to the series representation of $h$ and $\cL f_n$ (cf.~\eqref{mahmood}) we conclude that 
 \be\label{FR1}
 \begin{split}
 \|h -\cL f_n \|_\infty &  \leq 2 \e+
 \| \sum_{\{x,y\}\in \cE_S: \{x,y\}  \subset A_N } c_{x,y}  \left(T_{x,y} f -T_{x,y}f_n\right) \|_\infty\\
 &\leq 2 \e+
 \sum_{\{x,y\}\in \cE_S: \{x,y\} \subset A_N } c_{x,y} \|  T_{x,y} f  -  T_{x,y}f_n  \|_\infty\,.
 \end{split}
 \en
 For $n \geq N$ and $\{x,y\}\subset A_N\subset A_n$ we have that $\eta^{x,y}$ and $(\eta^{(n)})^{x,y}$ coincide on $A_n$ and also $\eta$ and $\eta^{(n)}$ coincide on $A_n$ and therefore by \eqref{eccolo8}
\be\label{FR2}
|T_{x,y} f(\eta)- T_{x,y} f^{(n)}(\eta)|\leq | f(\eta^{x,y})-f((\eta^{(n)})^{x,y})|+ | f(\eta)-f(\eta^{(n)})| \leq 2 \sum _{z\not \in A_n} \D_f(z) \,.
\en
By combining \eqref{FR1} and \eqref{FR2} and setting $C(N):=\sum_{\{x,y\}\in \cE_S: \{x,y\} \subset A_N } c_{x,y} $,  we conclude that $ \|h -\cL f_n \|_\infty  \leq 2 \e+ 2 C(N) \sum _{z\not \in A_n} \D_f(z) $ for $n\geq N$. Hence, by the bound $\vertiii{f}<+\infty$, $\limsup_{n\to +\infty}  \|h -\cL f_n \|_\infty  \leq 2\e$. By the arbitrariness of $\e$ we conclude that  $\lim_{n \to +\infty} \|h -\cL f_n \|_\infty =0$, thus proving the lemma.}

\begin{Remark}\label{batteria1000}
\rot{For later use we stress that, in the proof of Proposition \ref{eastpak}, for  any $f$ with 
$\vertiii{f}<+\infty$ and $\vertiii{f}_*<+\infty$   we have built a sequence $f_n \in \cC\subset \cD(\cL) $ such that $\|f-f_n\|_\infty\to 0$ and $\|\cL f -\cL f_n \|_\infty \to 0$ as $n\to +\infty$.}
\end{Remark}
\subsection{Proof of Proposition \ref{volpino}}\label{tana_volpino}
\rot{We set $\cW:=\{f\in C(\{0,1\}^S)\,:\, \vertiii{S(t)f}<+\infty \text{ and } \vertiii{S(t)f}_*<+\infty \;\forall t \in \bbQ_+\}$.
We start showing that if $\cW$ is a core then also $\cC$ is a core. If $\cW$ is a core, then   for any $h\in \cD(\cL)$ there exists a sequence $h_n \in \cW\subset \cD(\cL)$ such that $\| h_n - h\|_\infty \to 0 $ and $\| \cL h_n - \cL h\|_\infty \to 0$ as $n\to +\infty$. To prove that $\cC$ is a core, it is then enough to show that for any $f\in \cW$ there exists a sequence $f_n \in \cC$ such that $\| f_n - f\|_\infty \to 0 $ and $\| \cL f_n - \cL f\|_\infty \to 0$. This follows from the proof of Proposition \ref{eastpak} and in particular from Remark \ref{batteria1000}.}

\rot{Now it remains to  show that $\cW$ is a core for $\cL$.
 By a slight modification of \cite[Proposition~3.3]{EK}  it is enough to show the following: (i)
  $\cC $ is dense in $C(\{0,1\}^S)$; (ii) $\cC \subset \cW \subset \cD(\cL)$ and (iii) $S(t) f \in \cW$ for any $f\in \cC$ and $t\in \bbQ_+$. The difference with  \cite[Proposition~3.3]{EK}  is that in (iii)  we require that $S(t) f\in \cW$ only for $t\in \bbQ_+$ and not for all $t\in \bbR_+$.
  The reader can easily check that the short proof of \cite[Proposition~3.3]{EK}  still works. Indeed, the invariance of $\cW$ is used in the proof only to show that, given $f\in \cC$, 
  $ f_n:= \frac{1}{n}\sum_{k=0}^{n^2} e^{-\l k/n} S(k/n)f $ belongs to $\cW$.}

   \rot{Let us check the above properties (i), (ii) and (iii).
  Property (i) is well known  \cite{L1}.   The inclusion $\cW\subset \cD(\cL)$  in (ii) follows from Proposition  \ref{eastpak} since $\vertiii{f}=\vertiii{S(0)f}<+\infty$ and 
  $\vertiii{f}_*=\vertiii{S(0)f}_*<+\infty$ for any $f\in \cW$. If we prove that $\cC\subset \cW$ (thus completing (ii)), then automatically we have (iii)
  since, by definition of $\cW$,  $S(t)\cW\subset \cW$ for any $t\in \bbQ_+$. Hence, it remains to prove that $\cC\subset \cW$. To this aim we first show that, given $f\in C(\{0,1\}^S)$, 
  \be\label{batteria}
 \D_{S(t) f} (u) \leq \sum _{x\in S} \D_f(x)  P_x(X_t=u)\,,
 \en 
 where the probability $ P_x$ refers to the random walk starting at $x$.
 To prove \eqref{batteria} suppose that $\s$ and  $\xi$ in $\{0,1\}^S$  are equal except at $u$. Recall that, by the graphical construction, 
 $S(t) f(\s)=\bbE[ f( \eta_t ^\s[\cK])]$ and   $S(t) f(\xi)=\bbE[ f( \eta_t ^\xi[\cK])]$. Since 
 $ \eta_t ^\s[\cK](x)= \s( X_t^x[\cK])$ and  $ \eta_t ^\xi[\cK](x)= \xi( X_t^x[\cK])$, we get that 
  $ \eta_t ^\s[\cK](x)=  \eta_t ^\xi[\cK](x)$ for all $x$ such that $X_t^x[\cK]\not =u$. Therefore, by \eqref{eccolo8}, 
 \[
 | f( \eta_t ^\s[\cK])- f( \eta_t ^\xi[\cK])| \leq \sum _{x\in S} \D_f(x) \mathds{1}( X_t^x[\cK]=u)\,.
 \]
 It then follows that
 \[ 
 | S(t) f(\s)-S(t) f(\xi)| \leq \sum _{x\in S} \D_f(x) \bbP( X_t^x[\cK]=u)=\sum _{x\in S} \D_f(x)  P_x(X_t=u)\,.
 \]
 From the above  estimate we get   \eqref{batteria}.}

 \rot{From  \eqref{batteria} we then have  
\begin{align*} 
& \vertiii{S(t) f}=\sum_{u \in S} \D_{S(t)f}(u)\leq \sum_{u\in S}\sum _{x\in S} \D_f(x)  P_x(X_t=u)= \sum _{x\in S} \D_f(x) =\vertiii{f}\,,\label{solare95}\\
 & \vertiii{S(t)f}_* =\sum_{u\in S} c_u \D_{S(t) f} (u) \leq
\sum_{u\in S} c_u \sum _{x\in S} \D_f(x)  P_x(X_t=u)
=\sum _{x\in S} \D_f(x)  E_x\big[c_{X_t}\big]\,.
\end{align*}
From the above two bounds and \eqref{condo} we get immediately that $\cC\subset \cW$.
}
 

%
\subsection{Proof of Lemma \ref{sfinimento}}\label{lollo} We need to prove for any $f\in C_c(S)$ that 
\be\label{atzechi}\lim_{t\da 0} \sum_{x\in S}  \Big(  \frac{P_t f (x) -f(x) }{t} -  \tilde \bbL f(x) \Big)^2=0\,.
\en
Since functions $f\in C_c(S)$ are finite linear combinations of Kronecker's functions, it is enough to consider the case $f(x)=\mathds{1}(x=x_0) =\d_{x,x_0}$ for a fixed $x_0 \in S$. 
To this aim we  write $N$ for the number of jumps performed by the random walk  in the time window $[0,t]$.
 Recall that $P_t f(x)=E_x \big[ f(X_t) \big]$. 
 We write $P_x$ for the probability on the path space associated to the random walk starting at $x$.
 We have
 \begin{equation*}
 \begin{split}
 &  \frac{P_t f (x) -f(x) }{t} -  \tilde \bbL f(x)   =f(x) \Big[ \frac{P_x(N=0)-1}{t}-  c_x\Big] 
 \\
 & +\sum_{y} f(y) \Big[\frac{P_x( N=1,\, X_t=y)}{t}- c_{x,y}\Big]
  +\frac{ P_x( N\geq 2,\, X_t=x_0)}{t}\,.
 \end{split}
 \end{equation*} As a consequence (recall that $f(x)=\d_{x,x_0}$)
to get \eqref{atzechi}  it is enough to prove the following:
\begin{align}
& \lim_{t\da 0}   \Big[ \frac{ e^{-c_{x_0} t} -1}{t} - c_{x_0} \Big]^2 =0\,,\label{lim1}\\
& \lim_{t\da 0 } \sum _{x\in S} \Big[ c_{x,x_0} \Big(\frac{\int_0 ^t ds\, e^{-c_x s} e^{-c_{x_0} (t-s) } }{t} -1 \Big) \Big]^2  =0\label{lim2}\,,\\
& \lim_{t\da 0}\sum_{x\in S} \frac{1}{t^2}P_x\big( X_t =x_0, N \geq 2\big) ^2  =0\,.\label{lim3}
\end{align}
  \eqref{lim1} is trivial. 
  \eqref{lim2} can be rewritten as (recall that $c_{x,x_0}=c_{x_0,x}$) 
\[
\lim_{t\da 0 } \sum _{x\in S}   c_{x_0,x}^2  F_t (x) =0\,,\qquad F_t(x):=
\Big(\frac{\int_0 ^t ds\, e^{-c_x s} e^{-c_{x_0} (t-s) } }{t} -1 \Big) ^2 \,.
\]
Since $\| F_t \|_\infty \leq 1$ and $\lim_{t\da 0} F_t(x)=0$ for all $x\in S$,  the above limit follows from the dominated convergence theorem if we show that $\sum _{x\in S}   c_{x_0,x}^2<+\infty$. To this aim we observe that, since $c_{x_0}= \sum _{x\in S}   c_{x_0,x} <+\infty$, it must be $\sup_{x\in S} c_{x_0,x} <+\infty$. Then we can bound  $\sum _{x\in S}   c_{x_0,x}^2\leq c_{x_0} \sup_{x\in S} c_{x_0,x}<+\infty$.

To prove \eqref{lim3} we use  the symmetry of  the random walk to get the bound
\be\label{bagno}
\begin{split}
& \sum_{x\in S} P_x\big( X_t =x_0, N \geq 2\big) ^2=\sum_{x\in S} P_{x_0}\big( X_t =x, N \geq 2\big) ^2\\
&\leq \Big[ \sum_{x\in S} P_{x_0}\big( X_t =x, N \geq 2\big)  \Big]  P_{x_0}\big( N \geq 2\big)=P_{x_0}\big( N \geq 2\big)^2 \,.\end{split}
\en
Now we note that 
\[ \frac{1}{t} P_{x_0}\big( N \geq 2\big)  =\sum_{x\in S} c_{x_0,x} G_t(x)\,, \qquad G_t(x):=\frac{1}{t} \int_0 ^t ds \, e^{-c_{x_0} s}( 1- e^{-c_x (t-s) } )\,.
\]
We have $\| G_t\|_\infty  \leq 1$ and $\lim_{t\da 0} G_t(x) = 0$. Since in addition  $\sum_{x\in S} c_{x_0,x}=c_{x_0}<+\infty$, by the dominated convergence theorem we conclude that $\lim_{t\da 0} P_{x_0}\big( N \geq 2\big)/t=0$. As a byproduct of the above limit and \eqref{bagno} we get \eqref{lim3}.
 \section{Graphical construction and Markov generator of SEP: proofs}\label{sec_proof_SEP}
\rosso{We start with the proof of Lemma \ref{lemma_somma}}.


 \begin{proof}[Proof of Lemma \ref{lemma_somma}] 
 Trivially it is enough to prove the first bound. To this aim we argue by contradiction and assume Assumption SEP to hold and  that  $c_x^s =+\infty$  for some $x\in S$.
  Without loss of generality, we can suppose that  $x=s_1$ (see \eqref{enu}). 
  Due to the superposition property of independent Poisson point processes, \rosso{under $\bbP$},  $M_n (t):= \sum _{i=2}^n \cK^s_{x, s_i}(t)$, $t\geq 0$,  is  a Poisson process with parameter $\l_n:= \sum _{i=2}^n c^s_{x,s_i}$. 
 Hence, fixed $k\in \bbN$ and  $t>0$, 
 $\bbP( M_n (t) \leq  k ) =\sum_{j=0} ^k e^{- \l_n t} (\l_n  t) ^j / j!
 $. Since by hypothesis $\lim_{n\to +\infty} \l_n =c_x=+\infty$, we conclude that $\lim_{n \to +\infty} \bbP( M_n(t) \leq k )=0$.  Since in addition  the sequence of events $n\mapsto \{M_n (t) \leq k\}$ is monotone and decreasing, we get that  $\bbP( \cap _{n \geq 2} \{ M_n(t) \leq k \} )=0 $. Note that  $\cap _{n \geq 2} \{ M_n(t) \leq k \} $ means that $x$ has degree at most $k$ in the graph $\cG_{t}(\cK)$. As a consequence,  for any fixed $t>0$, $\bbP$--a.s. the vertex $x$ has infinite degree  in the graph $\cG_{t}(\cK)$, thus contradicting Assumption SEP.
 \end{proof}

%

Recall the set $\G_*\subset D_\bbN^{\cE_S^o}$ introduced in Definition \ref{def_gamma_o}.
\rot{Moreover, recall that sets  $B_{r,x}(\cK)$  and $\cC_{t,x}$ introduced in Definition \ref{brunetta}.}


  \begin{Remark}\label{rem_latte}
   By the graphical construction and since $\cK\in \G_*$,  $B_{r,x}(\cK)$ is a finite set and  $\{\cK\in \G_*\,:\, B_{r,x}(\cK)=B\}$ is a Borel subset of $D_\bbN^{\cE_S^o}$ for all $B\subset S$.
   Fix now $t$ with $rt_0<t\leq (r+1) t_0$. Then,
    for any $\s\in \{0,1\}^S$, the value   
   $\eta_t ^\s(x)[\cK]$ depends on $\s$ only through   the restriction of  $\s$ to $\rot{\cC_{t,x}(\cK)=}B_{r,x}(\cK)$ and, knowing that  $\rot{\cC_{t,x}(\cK)=}B_{r,x}(\cK)=B$, it  depends on $\cK$ on through \rot{the} values $\cK_{y,z}(s)$ with $y\not= z $ in $B$ and $s\in [0,t]$.
      \end{Remark}

 \rot{The following lemma is the SEP version of Lemma \ref{messi}}:
 \begin{Lemma}\label{messi_o} Fixed $t\geq 0$, the map $ \{0,1\}^S \times \G_* \ni (\s,\cK) \mapsto  \eta^\s_t [\cK]  \in \{0, 1\}^S$ is continuous in $\s$ (for fixed $\cK$) and measurable in $\cK$ (for fixed $\s$).  
 \end{Lemma}
 \begin{proof} For $t=0$ we have $ \eta^\s_t [\cK]=\s$ and the claim is trivially true. We fix  $t>0$ and take $r\in \bbN$ such that $r t_0 <t \leq (r+1) t_0$.

 We first  prove the continuity in $\s$ for fixed $\cK$. By \eqref{sarto}, 
  given  $N\in \bbN_+$,  we have  $d( \eta^\s_t [\cK], \eta^{\s'}_t [\cK])\leq 2^{-N}$ if $\eta^\s_t [\cK](s_n)= \eta^{\s'}_t [\cK](s_n)$ for all $n=1,2,\dots, N$. 
  Recall Definition \ref{brunetta} and set 
$A(\cK):= \cup _{n=1}^N B_{ r,s_n} (\cK)$. 
 By  Remark \ref{rem_latte}
 we have $\eta^\s_t [\cK](s_n)= \eta^{\s'}_t [\cK](s_n)$ for all $n=1,2,\dots, N$ whenever $\s$ and $\s'$ coincide on $A(\cK)$, which is automatically satisfied if $d(\s,\s')$ is small enough \rosso{since $A(\cK)$ is finite}. This proves the continuity in $\s$. 
  
   We  now prove the measurability in $\cK$ for fixed $\s$.  We just need to prove that, fixed $\s\in \{0,1\}^S$ and $x\in S$,  the set    $\{ \cK\in \G_*\,:\, \eta_t^\s[\cK](x)=1\}  $  is measurable in $D_{\bbN}^{\cE_S^o}$. 
   Given $B\subset S$ finite, call $\cA_B$ the event in $D_{\bbN}^{\cE^o_B}$ that the SEP on $B$ built by  the \rosso{standard} graphical construction on $B$ has value $1$ at $x$  at time $t$ \rosso{($\cE^o_B$ is given by the pairs $(y,z)$ with  $y\not = z$ in $B$)}. Since in this case  all is finite, $\cA_B$ is Borel in $D_{\bbN}^{\cE^o_B}$. We define $\cA_B^*$ as the Borel set in  $D_{\bbN}^{\cE^o_S}$ obtained as product set of $\cA_B$ with 
   $D_{\bbN}^{\cE^o_S\setminus \cE^o_B}$.
   Then, see Remark \ref{rem_latte}, 
   \[
   \{ \cK\in \G_*\,:\, \eta_t^\s[\cK](x)=1\} = \cup _{B\subset S \text{ finite}} \left( \{ \cK\in \G_*\,:\, \rosso{B_{r,x}(\cK)}=B\}\cap \cA_B^*\right) \,. \]
    By  Remark \ref{rem_latte} the set  $ \{ \cK\in \G_*\,:\, B_{r,x}(\cK)=B\}
 $ is Borel in $D_{\bbN}^{\cE_S^o}$, thus allowing to conclude.    
%
%
    \end{proof}
 
 \begin{Lemma}\label{pablito_o} For each $\s\in \{0,1\}$ and $\cK\in \G_*$, the path 
 $\eta^\s _\cdot[\cK] :=\bigl(  \eta^\s_t [\cK] \bigr)_{t\geq 0}$ belongs to $D_{\{0, 1 \}^S}$. Moreover, fixed $\s \in \{0,1\}^S$, the map
 \be\label{gallavotti_o}
  \G_* \ni \cK \mapsto  \eta^\s_\cdot [\cK]  \in D_{\{0,1\}^S} 
  \en
  is measurable in $\cK$.
 \end{Lemma}
 \begin{proof}
 Let us show that $\eta^\s _\cdot[\cK]$ belongs to $D_{\{0, 1 \}^S}$.  We start with the right continuity. Fix $t\geq 0$ and let $r t_0 <t \leq (r+1) t_0$ with $r\in \bbN \cup \{-1\}$.  Due to \eqref{sarto} we just need to show that  $\lim_{u\downarrow t} \eta^\s_u [\cK](x)=\eta^\s_t [\cK](x)$ for any $x\in S$. 
  By Remark \ref{rem_latte}, as $\cK \in \G_*$, the set $B:=B_{ r+1,x}(\cK)$   is finite
   and 
  $\eta^\s_u [\cK](x)=\eta^\s_t [\cK](x)$ if  $\cK_{y,z}$ has no jump time in \rosso{$(t,u]$} for all $(y,z)\in \cE_S^o$ with $y,z\in B$. Since $B$ is finite and the jump times of $\cK_{y,z}$ form a locally finite set, we conclude that the above condition is satisfied  for $u$ sufficiently close to $t$. This proves the right continuity. To prove that  $\eta^\s _\cdot [\cK]$ has  left limit in $t>0$, by \eqref{sarto} we just need to show that, for any $x\in S$,   $\eta^\s_\cdot [\cK](x)$  has  left limit in $t>0$, i.e. it is constant in the time window $(u,t)$ for some $u<t$. By Remark \ref{rem_latte} and since $\cK\in \G_*$ this follows from the fact that, by taking $u$ close to $t$,  for all $(y,z) \in \cE_S^o$ with $y,z\in B$ the path  $\cK_{y,z}$ has no jump in $(u,t)$ (as $B$ is finite).
  
  The proof of the measurability of the map  \eqref{gallavotti_o} follows the same steps of the proof of the measurability of the map  \eqref{gallavotti} (just replace there Lemma \ref{messi} by Lemma \ref{messi_o}). 
   \end{proof}

\subsection{Proof of Propositions \ref{costruzione_SEP} and  \ref{feller_SEP}}\label{provette}

The proof of Proposition \ref{costruzione_SEP} equals verbatim the proof  of  Proposition \ref{costruzione_SSEP}  by replacing  Lemma \ref{messi} with  Lemma \ref{messi_o}, with exception for the derivation  of \eqref{andreino1} (which can anyway be obtained from the graphical construction).
The proof of Proposition \ref{feller_SEP} equals verbatim the proof of Proposition \ref{feller_SSEP} by replacing  Lemma \ref{messi} with  Lemma \ref{messi_o}.


\subsection{Proof of Propositions \ref{inf_SEP}}\label{provette_bis}
The proof follows in good part \cite[Appendix~B]{F_SEP}. Since the notation there is slightly different and some steps have to be changed, we give the proof for completeness.

 Below   $\cK$ will always  vary in $\G_*$, without further mention.
Given $t\in (0,t_0]$  we denote by 
 $\cG _{t}(\cK)$ 
the undirected graph with vertex set $S$ and edge set 
  $\{ \{x,y\}\in \cE_S \,:\, \cK^s_{x,y}(t)  >0\}$.   As  $\cG _{t}( \cK)$ is a subgraph of $\cG _{t_0}(\cK)$,  the graph  $\cG _{t}(\cK)$  has only connected components of finite cardinality. Moreover, as $t\leq t_0$,     it is simple to check that $\eta_t^\s [\cK]$ can be obtained by the graphical construction of  Section \ref{sec_SEP}  but working with the graph  $\cG _{t}(\cK)$  instead of $\cG _{t_0}(\cK)$. 
  
  Let $f : \{0,1\}^S \to \bbR$ be a local function such  that $f(\eta)$ is defined in terms only of $\eta(x)$ with $x\in A$ and $A\subset S$ finite.  We set $\cE_{A}:=\{\{x,y\} \in \cE_S \,:\, \{x,y\} \cap A \not = \emptyset\}$. 
 By Lemma \ref{lemma_somma} we have 
 \be\label{stimetta_o}c^s_A  := \sum _{\{x,y\} \in \cE_{A}} c^s_{x,y} \leq 
 \sum_{x\in A} \sum _{y \in S} c^s_{x,y} = 
 \sum _{x\in A} c^s_x <+\infty\,.
 \en
 By \eqref{stimetta_o}  it is simple to check that 
  the r.h.s.'s of \eqref{mammaE_bis} is an absolutely convergent series in $C(\{0,1\}^S)$ defining therefore a  function in $C(\{0,1\}^S)$  that we denote by $\hat \cL  f$. Hence we just need to prove that $\cL  f= \hat \cL f$.

 Since  $\cK^s_A(t):= \sum_{\{x,y\} \in \cE_A} \cK^s_{x,y}(t)$ is a Poisson random variable with finite  parameter  $c_A$, we have 
  $\bbP ( \cK\,:\,  \cK^s_A(t) \geq 2) = 1- e^{-c_A  t} (1+ c_A  t) \leq C  t^2$.
 When  $\cK_A(t)=1$, we define  the pair  $\{x_0,y_0\}$ as the only edge in $\cE_A$ such that $\cK_{x_0,y_0}(t)=1$. To have a univocally defined labelling, as in the proof of Proposition \ref{inf_SSEP}, 
if  the pair has only one point in $A$, then we call this point $x_0$ and the other one $y_0$. Otherwise, we call $x_0$ the minimal (w.r.t. the enumeration \eqref{enu}) point inside the pair. 
\begin{Claim}\label{violaceo}
Let $F $ be the event that (i) $\cK^s_A(t)=1$ and  (ii) $\{x_0, y_0\}$ is not a  connected component of $\cG_t (\cK)$. Then $\bbP (F) =o(t)$.
\end{Claim}
\begin{proof}[Proof of Claim \ref{violaceo}] 
We first show that   $F\subset G$, where 
\[
G=\bigl\{ \cK^s_A(t)=1\,, \; x_0\in A\,,\; y_0\not \in A\,,\;\\
 \exists z \in S \setminus (A \cup \{y_0\}) \text{ with } \cK^s_{y_0,z}(t)\geq 1\bigr\}\,.
\]

To this aim suppose first that $\cK^s_A(t)=1$ and 
   $x_0,y_0\in A$. Then $\{x_0,y_0\}$ must be a connected component in $\cG_t (\cK)$ (otherwise we would contradict $\cK^s_A(t)=1$).  Hence, $F$ implies that 
   $x_0\in A$ and $y_0 \not \in A$. By $F$, $\{x_0,y_0\}$  is   not a connected component of $\cG_t (\cK)$, and therefore there   exists
a point $z\in S\setminus\{x_0,y_0\}$  such that $\cK^s_{x_0,z}(t) \geq 1$ or $\cK^s_{y_0,z}(t)\geq 1$. The first case cannot  occur as $ \cK^s_A(t)=1$, $x_0\in A$ and $\cK^s_{x_0,y_0}(t) \geq 1$. By the same reason, in the second case it must be $z\not \in A$.  Hence, there exists
$z\in S \setminus (A\cup\{y_0\})$ such that $\cK^s_{y_0,z}(t)\geq 1$, thus concluding the proof that 
$F\subset G$.

As $F\subset G$  to prove that $\bbP(F)=o(t)$ it is enough to show that $\bbP(G)=o(t)$.  To this aim we first estimate $\bbP(G)$ by \begin{equation*}
\begin{split}
 \bbP (G) & \leq \sum_{x\in A} \sum _{y \in S\setminus A}  \bbP (  \cK^s_{x,y}(t)=1\,, \sum_{z\in S \setminus (A\cup\{y\})} \cK^s_{y,z}(t) \geq 1) \\
& \leq  t \sum_{x\in A} \sum _{y \in S } c^s_{x,y}   e^{- c^s_{x,y} t}  (1- e^{-c^s_y t})\,.
\end{split}
\end{equation*}
By applying the dominated convergence theorem   to the measure giving weight $c_{x,y}^s$ to the pair $(x,y)$ with $x\in A$ and $y\in S$ and by using 
 \eqref{stimetta_o},  we get $\lim _{t\da 0} \bbP (G)/t=0$.  \end{proof}

We define $H $ as the event that (i) $\cK^s_A(t)=1$ and  (ii) $\{x_0, y_0\}$ is a connected component of $\cG_t ( \cK)$.  Since  $\bbP ( \cK^s_A(t) \geq 2)  \leq C t^2$ and due to 
 Claim \ref{violaceo}  we get 
\be \label{eccolo}
\bbP ( \{\cK^s_A(t)=0\} \cup H ) = 1-o(t) \,.
\en
We set $\cE_A^o:=\{ (x,y)\in \cE_S^o \,:\, \rosso{\{x,y\}\in \cE_A} \}$. Given $(x,y)\in \cE_A^o$ we set 
\[ W_{x,y}:=  \{\cK^s_A(t)=1 \text{ and }  \{x_0,y_0\}
=\{x,y\}\,\}\cap \{ \cK_{x,y}(t)=1\}\, .
\,
\]

Trivially, by the graphical construction,  if $\cK^s _A(t)=0$, then $\eta^\s_t[\cK](z)= \s(z)$ for all $z\in A$.  On the other hand, if \rosso{$W_{x,y}\cap H$} takes place, then $\eta^\s_t[\cK](z)= \s^{x,y}(z)$ for all $z\in A$ if $\s(x)=1$ and $\s(y)=0$, otherwise $\eta^\s_t[\cK](z)= \s(z)$ for all $z\in A$.  \rosso{By \eqref{eccolo} and the above observations and 
since $S_t f (\s):= \bbE\left[ f\big( \eta^\s_t [\cK] \big)\right]$}, we can write
\[
S(t) f (\s) -f(\s) =\sum_{ (x,y)\in \cE^o_A} \s(x) (1-\s(y))[   f(\s^{x,y}) -f (\s) ]\bbP ( H\cap W_{x,y}) + o(t)\,.
\]
As $\bbP(F)=o(t)$  and $f$ is bounded
we can rewrite  the above r.h.s. as
 \begin{equation*}
\begin{split}&
\sum_{ (x,y)\in \cE^o_A} \s(x) (1-\s(y))[   f(\s^{x,y}) -f (\s) ]\bbP ( W_{x,y}) + o(t)\\
&=t\sum_{ (x,y)\in \cE^o_A} \s(x) (1-\s(y))[   f(\s^{x,y}) -f (\s) ] c_{x,y} e^{- c_A t}+  o(t)\,.
\end{split}
\end{equation*}
As  $\lim_{t\da 0} o(t)/t = 0 $  uniformly in $\s$, 
by the dominated convergence theorem  (use \eqref{stimetta_o}) we can conclude that $\hat \cL_\o f = \cL_\o f $.

\subsection{Proof of Proposition \ref{eastpakbis}}\label{cedrata} \rot{The proof can be obtained by   slight modifications from  the proof of Proposition \ref{eastpak} for the symmetric case. We give it for the reader's convenience.     
Let  $f$ be as in Proposition \ref{eastpakbis} and let  $T_{x,y} f (\eta):=f(\eta^{x,y})-f(\eta)$. 
Since   $|f(\eta^{x,y})-f(\eta)|\leq \D_f(x)+\D_f(y)$ by \eqref{eccolo8}, $\| T_{x,y} f \|_\infty \leq \D_f(x)+\D_f(y)$ and 
  \be\label{bici88bis} 
  \begin{split}
 &  \sum_{x\in S}\sum_{y\in S}\| c_{x,y} \eta(x) (1-\eta(y)) T_{x,y} f \| _\infty\leq 
  \sum_{\{x,y\}\in \cE_S} c^s_{x,y} \|T_{x,y} f\|_\infty\\
  &  \leq \sum_{\{x,y\}\in \cE_S} c^s_{x,y}(\D_f(x)+\D_f(y))=\sum_{x\in S} c^s_x \D_f(x)\leq \vertiii{f}_\star<+\infty\,.
  \end{split}
  \en
  This proves the absolute convergence of  $ \sum_{x\in S}\sum_{y\in S}  c_{x,y} \eta(x) (1-\eta(y)) T_{x,y} f(\eta) $ as series of functions in $C(\{0,1\}^S)$. Let us call $h(\eta)$ the resulting function. }
  
  \rot{Since $\cL$ is close, 
   to conclude it is enough to exhibit a sequence of local functions $f_n$ such that $\| f_n -f\|_\infty\to 0$ and $\|  \cL f_n - h\|_\infty \to 0$ as $n\to +\infty$.
 To this aim define  $A_n$, $\eta^{(n)}$ and $f_n(\eta):=f(\eta^{(n)})$ as in the proof of Proposition \ref{eastpak}.  
  By what proved there, we know that 
  $f_n$ is a local function and 
 $\| f- f_n\|_\infty  \leq \sum _{x\not \in A_n}  \D_f (x) $, hence it goes  to zero as $n\to +\infty$ since $\vertiii{f}<+\infty$.  
It remains to  prove  that 
 $\|  \cL f_n - h\|_\infty \to 0$.  We fix $\e>0$. Due to \eqref{bici88bis} 
 we can fix $N$ such that
 \be \label{ululato1bis}
  \sum_{
  \substack{(x,y)\in S\times S:\\ \{x,y\} \not \subset A_N}
   }\| c_{x,y} \eta(x) (1-\eta(y)) T_{x,y} f \| _\infty  \leq 
 \sum_{
 \substack{
 \{x,y\} \in \cE_S: \\ \{x,y\}\not \subset A_N 
 } }
  c^s_{x,y} \|T_{x,y} f\|_\infty
\leq \e\,.
\en
 Since  $\| T_{x,y} f_n \|_\infty \leq \| T_{x,y} f \|_\infty$,  \eqref{ululato1bis} holds also with $f_n$ instead of $f$ for any $n=1,2,\dots$.
 As a consequence,
  due to the series representation of $h$ and $\cL f_n$ (cf.~\eqref{mammaE_bis}) and due to \eqref{FR2}, we conclude that  for $n\geq N$
 \be\label{FR1bis}
 \begin{split}
 \|h -\cL f_n \|_\infty  & \leq 2 \e+
 \sum_{\{x,y\}\in \cE_S: \{x,y\} \subset A_N } c^s_{x,y} \|  T_{x,y} f  -  T_{x,y}f_n  \|_\infty\\
 &    \leq 2 \e+ 2
 \sum_{\{x,y\}\in \cE_S: \{x,y\} \subset A_N } c^s_{x,y} \sum_{z\not \in A_n} \D_f (z)\,. 
 \end{split}
 \en
Setting $C(N):=\sum_{\{x,y\}\in \cE_S: \{x,y\} \subset A_N } c^s_{x,y} $,  we have proved that $ \|h -\cL f_n \|_\infty  \leq 2 \e+ 2 C(N) \sum _{z\not \in A_n} \D_f(z) $ for $n\geq N$. Hence, by the bound $\vertiii{f}<+\infty$ and the arbitrariness of $\e$,   $\lim_{n \to +\infty} \|h -\cL f_n \|_\infty =0$.}

\begin{Remark}\label{batteria1000bis}
\rot{For later use we stress that, in the above proof, for    any $f$ with 
$\vertiii{f}<+\infty$ and $\vertiii{f}_\star<+\infty$   we have built a sequence $f_n \in \cC\subset \cD(\cL) $ such that $\|f-f_n\|_\infty\to 0$ and $\|\cL f -\cL f_n \|_\infty \to 0$ as $n\to +\infty$.}
\end{Remark}
%

\subsection{Proof of Proposition \ref{volpinobis}}\label{astro}
\rot{We set $\cW:=\{f\in C(\{0,1\}^S)\,:\, \vertiii{S(t)f}<+\infty \text{ and } \vertiii{S(t)f}_\star<+\infty \;\forall t \in \bbR_+\}$. As in  the proof of Proposition \ref{volpino}, due to  Remark \ref{batteria1000bis}, we get that 
if $\cW$ is a core then also $\cC$ is a core. Moreover, due to  \cite[Proposition~3.3]{EK},   to  show that $\cW$ is a core  it is enough to prove the following: (i)
  $\cC $ is dense in $C(\{0,1\}^S)$; (ii) $\cC \subset \cW \subset \cD(\cL)$ and (iii) $S(t) f \in \cW$ for any $f\in \cC$ and $t\in \bbR_+$. Having Proposition \ref{eastpakbis}, the only nontrivial property to check is that $\cC\subset \cW$. So we focus on this.}
   
\rot{First we show
 that, for any $f\in C(\{0,1\}^S)$ and $x\in S$,  it holds  
 \be\label{cannolo}
 \D_{S(t) f}(x) \leq \sum_{y\in S} \D_f(y) \bbP ( x \in  \cC_{t,y}(\cK))\,.
 \en
 Given $x\in S$ let $\xi, \s\in \{0,1\}^S$ be equal expect at most at $x$.
 Then, by denoting by $\bbE$ the expectation w.r.t. $\bbP$,  we get 
  \be\label{checco1}
 \begin{split}
 \left| S(t) f (\xi) - S(t) f (\s)\right | &= \big | \bbE [ f (\eta^\s_t[\cK])]-\bbE  [f (\eta^\xi_t[\cK])] \big| \\
 & \leq 
 \bbE\big[ | f (\eta^\s_t[\cK])- f (\eta^\xi_t[\cK])| \big] \,.
  \end{split}
 \en
 By Remark \ref{rem_latte} and the graphical construction of SEP,  we get that 
  $\eta^\s_t[\cK](y)=\eta^\xi_t[\cK](y)$  if  $x \not \in C_{t,y}(\cK)$. Hence, by  \eqref{eccolo8},  we can bound the last term in \eqref{checco1} by 
 \be\label{checco2}
 \bbE\big[ \sum _{y : x \in \cC_{t,y}(\cK)} \D_f (y)\big]=\sum_{y\in S} \D_f(y) \bbP( x \in  \cC_{t,y}(\cK))\,. 
 \en 
 By combining \eqref{checco1} and \eqref{checco2} we have 
 \be\label{checco3}
 \D_{S(t)f}(x) \leq \sum_{y\in S} \D_f(y) \bbP( x \in  \cC_{t,y}(\cK))\,.
 \en
 From \eqref{checco3} we get 
 \begin{align}
&  \vertiii{S(t) f} \leq \sum _{x\in S}  \sum_{y\in S} \D_f(y) \bbP( x \in  \cC_{t,y}(\cK))= \sum _{y\in S} \D_f(y) \bbE\big[ | \cC_{t,y}(\cK)|\big]\,,\label{checco4}\\
 & \vertiii{S(t) f}_\star \leq \sum _{x\in S}  c_x ^s \sum_{y\in S} \D_f(y) \bbP( x \in  \cC_{t,y}(\cK))= \sum _{y\in S} \D_f(y) \bbE\big[ \sum_{x\in \cC_{t,y}(\cK)} c_x^s \big]\,.\label{checco5}
 \end{align}
 Given $f\in \cC$ we have $\D_f(y)=0$ for all $y$ except a finite set. Hence,  we get 
 $ \vertiii{S(t) f}<+\infty$
  due to \eqref{neve1} and \eqref{checco4}, while  $\vertiii{S(t) f}_\star<+\infty$ due to  \eqref{neve2} and \eqref{checco5}. This completes the proof that $\cC\subset \cW$.
  }
   \subsection{Proof of Proposition \ref{boreale}}\label{aurora85}
   \rot{As already observed after Proposition \ref{volpinobis}, \eqref{neve1} is equivalent to \eqref{danza1} and \eqref{neve2} is equivalent to \eqref{danza2}. Both \eqref{danza1} and \eqref{danza2} can be rewritten as 
      \be\label{boss}
\bbE\big[ \sum_{z\in B_{r,x}}  a _z\big]<+\infty \qquad \forall r\in \bbN\,,  \; x\in S\,,
\en
 where $a_z:=1$ for \eqref{danza1} and $a_z:=c_z^s$ for \eqref{danza2}.}
 
\rot{By definition of $B_{r,x}$, given $x_0\in S$   we have the following equality of events (where $\cK$ is the random object):
\be\label{goccia1}
\{x_0\in B_{r,x}\}= \cup_{x_1,x_2,\dots, x_r\in S} \left( F^0_{x_0,x_1}\cap F^1_{x_1,x_2} \cap \cdots \cap F^{r}_{ x_{r}, x}\right)\,,
\en
where $F^0_{a,b}:= \{ a\longleftrightarrow b\text{ in } \cG^0_{t_0}\}$, $F^1_{a,b}:= \{ a\longleftrightarrow b\text{ in } \cG^1_{t_0}\}$ and so on.
We point out that when $r=0$ \eqref{goccia1} has to be thought of as 
$\{ x_0 \in B_{0,x}  \}
=  F^0_{x_0,x}$.
Since under $\bbP$ the graphs $\cG^0_{t_0}$, $\cG^1_{t_0}$, $\cG^2_{t_0}$,... are i.i.d. and since $\bbP( F^i_{a,b})=p(a,b)$, by 
\eqref{goccia1} and a union bound we have 
\be\label{goccia2}
\bbP\big( x_0 \in B_{r,x}\big) \leq \sum_{x_1 ,x_2,\dots, x_{r}\in S} p(x_0,x_1) p(x_1,x_2) \dots p(x_{r-1}, x_r)p(x_r,x)\,.
\en
Hence we get
\be
\begin{split}
& \bbE\big[ \sum_{x_0 \in B_{r,x}}  a _{x_0}\big]  =\sum_{x_0\in S}\bbP\big( x_0 \in B_{r,x}\big)a_{x_0}\\
&\leq
\sum_{x_0,x_1, x_2\dots, x_r\in S} a_{x_0}p(x_0,x_1) p(x_1,x_2) \dots p(x_{r-1}, x_r)p(x_r,x) \\
& =\sum_{y_{r}, y_{r-1},\dots, y_0\in S}a_{y_r} p(y_r,y_{r-1}) p(y_{r-1},y_{r-2}) \dots p(y_1, y_0)p(y_0,x) \,.
\end{split}
\en
Trivially, by setting $n=r+1$  and $w_{i+1}:=y_i$ and  using  that $p(a,b)=p(b,a)$, the last sum can be written as 
$\sum_{w_1, w_2, \dots, w_n\in S} p(x,w_1)p(w_1,w_2)\cdots p(w_{n-1}, w_{n}) a_{w_{n}}$.}

\rot{By the above observations, we get \eqref{neve1} if \eqref{vento1} holds for all $x\in S$ and $n\in \bbN_+$ and similarly we get \eqref{neve2} if \eqref{vento2} holds for all $x\in S$ and $n\in \bbN_+$.
Finally, it is trivial to check that \eqref{vento1} holds all $x\in S$ and $n\in \bbN_+$ if and only if
\eqref{vento1bis} holds all $x\in S$ and $n\in \bbN$. The same equivalence can be easily checked  for \eqref{vento2} and \eqref{vento2bis}. }

\appendix

\section{Derivation of \eqref{albero}}\label{lingotto}
In this appendix we show that Conditions (3.3) and (3.8) in \cite[Chapter~I.3]{L1} are both equivalent to \eqref{albero}.
In the notation of \cite[Chapter~I.3]{L1}, given $x\not=y$ in $S$, one defines the measure  $c_{\{x,y\}}(\eta, d\z)$ on $\{0,1\}^{\{x,y\}}$ as $c_{x,y}\eta(x)(1-\eta(y)) \d_{(0,1)}(d\z)+ c_{y,x} (1-\eta(x))\eta(y) \d_{(1,0)}(d\z)$, and one sets $c_{T}(\eta, d\z):=0$ for $T\subset S$ with $|T|\not=2$.

  Condition (3.3)  in \cite[Chapter~I.3]{L1}  
  is given by $\sup_{x\in S} \sum_{T\ni x}c_T<+\infty$, where $ c_T= \sup\left \{ c_T( \eta, \{0,1\}^T)\,:\, \eta \in \{0,1\}^S\right\}$. In our case we have 
 \[
  c_T=  \begin{cases} \max\{c_{x,y}, c_{y,x} \} & \text{if }  T=\{x,y\} \text{ for some } x\not =y\,,\\
      0 & \text{otherwise}\,,
    \end{cases}
  \]
  and therefore the above mentioned condition is equivalent to \eqref{albero}.

  Given $u\in S$ and $T\subset S$, as in \cite[Chapter~I.3]{L1}  we define
  \[ c_T(u):=\sup\left\{\| c_T(\eta_1, d\z) -c_T(\eta_2,d\z)\|_{TV}: \eta_1 (z)=\eta_2(z) \text{ for all }z\not=u\right\}\,,
    \]
    where $\|\cdot\|_{TV}$ denotes the total variation norm.
    In our case, if $|T|\not =2$ then $c_T(u)=0$. If $T=\{x,y\}$ with $x\not=y$ in $S$, then
    \[
      c_T(u)= \begin{cases}  \max\{c_{x,y}, c_{y,x} \}  & \text{if } u\in T\,,\\
                0 & \text{otherwise}\,.
              \end{cases}
            \]
 Since Condition (3.8) in  \cite[Chapter~I.3]{L1}  is given by $\sup_{x\in S}\sum_{T\ni x} \sum_{u\not=x}c_T(u)<+\infty$, also this condition is equivalent to \eqref{albero}.

\bigskip

{\bf Acknowledgements}. \rot{I thank the anonymous referee for the corrections and comments}. I am very  grateful to R\'ath Bal\'azs  for the  stimulating discussions during the Workshop ``Large Scale Stochastic Dynamics'' (2022) at  MFO. I kindly acknowledge also  the organizers of this event.  
I thank the very heterogeneous population  (made of different animal and vegetable   species)  of my homes in Codroipo and Rome, where this  work has been written. 


\end{document}